\documentclass[reqno]{amsart}
\usepackage{amsmath,amssymb,amsfonts}
\usepackage{esint, cancel}
\usepackage{amsthm}
\usepackage{bbm}
\usepackage{todonotes}
\usepackage{mathrsfs}
\usepackage{geometry}
\usepackage{verbatim}
\usepackage{enumerate}
\usepackage{graphicx}
\usepackage{color}
\usepackage[utf8]{inputenc}
\usepackage[T1]{fontenc}
\usepackage{lmodern}
\usepackage[colorlinks=true,urlcolor=blue, citecolor=red,linkcolor=blue,
linktocpage,pdfpagelabels, bookmarksnumbered,bookmarksopen]{hyperref}
\usepackage[hyperpageref]{backref}
\usepackage{mathtools}
\mathtoolsset{showonlyrefs=true}

\newtheorem{theorem}{Theorem}[section]
\newtheorem{lemma}[theorem]{Lemma}

\newtheorem{corollary}[theorem]{Corollary}

\theoremstyle{definition}
\newtheorem{remark}[theorem]{Remark}

\newtheorem{definition}[theorem]{Definition}

\newtheorem*{ack}{Acknowledgments}

\newcommand{\R}{\mathbb{R}}
\newcommand{\E}{\mathcal{E}}

\numberwithin{equation}{section}

\title[Lane-Emden equation with nonlocal Neumann conditions]{Existence and non-existence results for a fractional Lane-Emden equation with nonlocal Neumann conditions}
\author[Cinti]{Eleonora Cinti}
\address[E.\ Cinti]{Dipartimento di Matematica
	\newline\indent
	Alma Mater Studiorum  Universit\`a di Bologna
	\newline\indent
	piazza di Porta San Donato, 5
	40126 Bologna, Italy}
\email{eleonora.cinti5@unibo.it}
\author[Talluri]{Matteo Talluri}
\address[M.\ Talluri]{Dipartimento di Matematica
	\newline\indent
	Alma Mater Studiorum  Universit\`a di Bologna
	\newline\indent
	piazza di Porta San Donato, 5
	40126 Bologna, Italy}
\email{matteo.talluri@unibo.it}
\author[Weth]{Tobias Weth}
\address[T.\ Weth]{Institute f\"ur Mathematics
	\newline\indent
	Goethe-Universit\"at Frankfurt
	\newline\indent
	Robert-Mayer-Strasse 10, 60629, Frankfurt am Main, Germany}
\email{weth@math.uni-frankfurt.de}

\begin{document}
\raggedbottom

    \begin{abstract}
    We consider the fractional Lane-Emden equation with a nonlocal Neumann condition in a half-space. We establish the existence of non-constant solutions for the critical problem  in any dimension. On the other hand, we show that, in dimension $n=1$, the subcritical problem admits only the trivial solution. This result follows from a new Pohozaev-type identity, which we obtain by using suitable decay estimates for the solutions.
    \end{abstract}
\maketitle
\section{Introduction}

For $s\in(0,1)$ and $1< p\le  \frac{n+2s}{n-2s}$, we consider the following nonlocal Neumann problem 
\begin{equation}\label{P}
\begin{cases}
(-\Delta)^s u=u^{p}\quad&\mbox{in }\R^{n}_+,\\
u\ge 0\quad&\mbox{in }\R^{n}_+,\\
\mathcal N_s u=0\quad&\mbox{in }\mathbb R^n_-,
\end{cases}
\end{equation}
where $\R^{n}_+:=\{(x_1,\cdots,x_{n})\in \R^{n}\,|\, x_{n}>0\}$ and $\R^{n}_-:=\{(x_1,\cdots,x_{n})\in \R^{n}\,|\, x_{n}<0\}$ denote, respectively, the positive and negative (open) half-spaces. Here  $(-\Delta)^s$ denotes the fractional Laplacian
\begin{equation}\label{FL}
(-\Delta)^su(x):=c_{n,s}\,\mathrm{PV}\int_{\mathbb R^n}\frac{u(x)-u(y)}{|x-y|^{n+2s}}dy,
\end{equation}
where $c_{n,s}$ is a normalization constant, and $\mathcal N_s$ is the following nonlocal normal derivative
\begin{equation}\label{Neu}
\mathcal N_s u(x):=c_{n,s}\int_{\R^{n}_+}\frac{u(x)-u(y)}{|x-y|^{n+2s}}dy\quad\mbox{for all }x\in\mathbb R^n_-,
\end{equation}
first introduced in \cite{DRV}. Observe that such a nonlocal Neumann condition is the natural one so that problem \eqref{P}  has a variational structure.
{Indeed, if $u,v: \R^n\to \R$ are bounded $C^2$-functions,} then the following integration by parts formula holds true for any open subset $\Omega \subset \R^n$, where $\Omega^c=\R^n\setminus \Omega$ (see \cite[Lemma 3.3]{DRV}):
\begin{equation}\label{int-by-parts}
\begin{aligned}
\frac{c_{n,s}}{2}\iint_{\R^{2n}\setminus(\Omega^c)^2}\frac{(u(x)-u(y))(v(x)-v(y))}{|x-y|^{n+2s}}&\,dx\,dy
=\int_\Omega v\,(-\Delta)^su\,dx + \int_{\Omega^c}v\,\mathcal N_s u\,dx.
\end{aligned}
\end{equation}

 The aim of this paper is twofold. On one hand, we show that
problem \eqref{P} admits at least one non-trivial solution in the critical case $p=\frac{n+2s}{n-2s}$. On the other hand, in dimension $n=1$ and for subcritical powers $1< p<\frac{1+2s}{1-2s}$, we establish a Liouville-type result, which states that any bounded solution is necessarily identically zero.
{In order to present our main results in detail}, let us introduce some notation.

For every $u:\R^{n}\to \R$ measurable, $s\in (0,1)$, and $\Gamma\subseteq \R^{2n}$, we define \[
    \E(u\:,\:\Gamma)=\iint_{\Gamma}\frac{\left(u(x)-u(y)\right)^2}{|x-y|^{n+2s}}dxdy
    .\]
   For $n>2s$ 
   we set $T(\R^{n}_+)=\R^{2n}\setminus \left(\R^{n}_-\right)^2$ and we define the space $H^s_{\R^{n}_+}$ as 
\[
    H^s_{\R^{n}_+}=\left\{u:\R^n\to\R\: \text{measurable}\::\:u\in L^{2^*_s}(\R^{n}_+)\text{ and }\E(u\:,\:T(\R^{n}_+))<+\infty\right\},
    \]where $2^*_s=\frac{2n}{n-2s}$ denotes the critical Sobolev exponent.
    We endow such a space with the norm \[
    \|u\|_{H^s_{\R^{n}_+}}=\|u\|_{L^{2^*_s}(\R^{n}_+)}+\E\left(u\:,\:T(\R^{n}_+)\right)^\frac{1}{2}.
    \]
The space $H^s_{\R^{n}_+}$ is the natural energy space associated to problem \eqref{P}. We can now give the notion of weak solution for our problem.
{\begin{definition}
    We say that a function $u \in H^s_{\R^{n}_+}$ is a weak solution to \eqref{P} if for any $v \in H^s_{\R^{n}_+}$ one has
    \[\iint_{T(\R^n_+)}\frac{(u(x)-u(y))(v(x)-v(y))}{|x-y|^{n+2s}}\,dx\,dy=\int_{\R^{n}_+}u^{p}v\,dx.\]
\end{definition}}

We can now state our first result.
 \begin{theorem}\label{existence-critical}Let $s\in(0,1)$ and $n>2s$. Then, there exists {a non-constant weak solution $u\in H^s_{\R^n_+}$} to \begin{equation}\label{P-critical}
    \begin{cases}
        (-\Delta)^su=u^{\frac{n+2s}{n-2s}}\quad&\text{ in }\R^n_+,\\
        u\ge 0&\text{ in }\R^n_+,\\
        \mathcal{N}_su=0\quad&\text{ in }\R^n_-.
    \end{cases}
    \end{equation}
\end{theorem}
    
    In our existence result, the fractional Sobolev constant in the whole $\R^n$ and the Sobolev constant for the Neumann problem in the half-space will play a central role. Let us introduce the corresponding notation.
    
    We define the \emph{Sobolev constant of $\R^n$} as 
    \begin{equation}\label{aubin talenti}
    S(\R^n):=\inf _{0 \neq u \in \mathcal D^s(\mathbb{R}^n)} \frac{\E(u\:,\:\R^{2n})}{\left(\int_{\R^{n}}|u|^{2^*_s} d x\right)^{\frac{2}{2^*_s}}},
    \end{equation}
where $\mathcal D^s(\R^n)$ is the abstract completion of $C^\infty_c(\R^n)$ with respect to $\mathcal E (\cdot, \R^{2n}).$
   
    It is well known that $S(\R^n)$ is achieved  by multiplications, dilations, and translations of the so-called \emph{Talenti Function}
    \begin{equation}\label{Talenti}U(x)=c_{n,s} \bigg(\frac{1}{1+\left|x \right|^2}\bigg)^{\frac{n-2s}{2}}.
    \end{equation}
    
    Moreover, we introduce the \emph{Neumann Sobolev constant of  $\R^n_+$} as \[
    S_\mathcal{N}(\R^n_+):=\inf _{0 \neq u \in H^s_{\mathbb{R}^n_+}} \frac{\E(u\:,\:T(\R^n_+))}{\left(\int_{\R^n_+}|u|^{2^*_s} d x\right)^{\frac{2}{2^*_s}}}.
    \]

   It is easy to see that problem \eqref{P-critical} admits a non trivial solution if the infimum in the definition of $S_\mathcal N(\R^n_+)$ is a minimum. In this case any minimizer is a solution. In \cite[Theorem 1.3]{MN}, Musina and Nazarov proved that
$S_\mathcal{N}(\R^n_+)$ is achieved if $S_\mathcal{N}(\R^n_+)<S(\R^n)$. Moreover, in \cite[Theorem 4.7]{MN}, they show that such a sufficient condition is satisfied for $s$ sufficiently close to $1^-$ if $n\ge 2$ and $s$ sufficiently close to $1/2^-$ if $n=1$. In the present paper, we prove that the above strict inequality holds true for all values of $s$ and $n$, according to the following result:

\begin{theorem}\label{strict inequality}
 For any $s\in(0,1)$ and {$n>2s$} we have \[
 S_\mathcal{N}(\R^n_+)<S(\R^n).
 \]
\end{theorem}
 The proof of such a result is based on the construction of a suitable competitor for $S_\mathcal N(\R^n_+)$, which crucially uses the fractional Talenti function \eqref{Talenti}.

  
  Finally, Theorem \ref{existence-critical} immediately follows from Theorem \ref{strict inequality} combined with \cite[Theorem 1.3]{MN}.
 \medskip

 Let us move now to the subcritical regime. The following is our second main result.
\begin{theorem}\label{non exitence result}
    Assume $n=1$, $0<s<\frac12$ and let $u \in H^s_{\R_+}\cap \
    L^\infty(\R_+)$ be a positive weak solution of \eqref{P} with $$
1< p<\frac{1+2 s}{1-2 s}.
$$
Then $u \equiv 0$.
\end{theorem}

In the classical local setting, Gidas and Spruck in \cite{GS} proved a celebrated Liouville-type result which states that if $1< p<\frac{n+2}{n-2}$, the only non-negative solution to the equation $-\Delta u=u^p$ in $\R^n$, is $u\equiv 0 $. If we consider the same local equation in the positive half-space with Neumann boundary conditions, it is easy to see that the same result holds true. Indeed, in such a case, thanks to the Neumann condition, we can extend evenly the solution to the whole $\R^n$ and apply the previous result.
Such Liouville-type results are not only interesting in themselves, but are also crucial tools when one wants to prove a priori estimates for solutions to certain nonlinear equations by a blow up argument à la Gidas-Spruck (see \cite{GS2}).

In the fractional setting, an analogous Liouville-type result for the subcritical Lane-Emden equation in the whole $\R^n$ has been established in \cite{BCPS,JLX}. When considering the problem in a half-space (or in more general subsets of $\R^n$), the only available results consider Dirichlet boundary conditions. In particular, it has been proven in \cite{Chen,Fall-Weth1} that, if $1< p< \frac{n+2s}{n-2s}$, any non-negative solution to 
$$\begin{cases}
    (-\Delta)^su=u^p & \mbox{in }\R^n_+\\
    u= 0 &  \mbox{in }\R^n_-
\end{cases}$$
identically vanishes. This result has been improved in \cite{Fall-Weth2,Quuas-Xia} by allowing any power $1< p < \frac{n+2s-1}{n-2s-1}$. The proof of such results is based on the method of moving planes.

 Regarding similar Liouville-type theorems for problem \eqref{P}, the situation is much more delicate, due to the non-local nature of the non-local Neumann condition. Indeed, in this case, it is not possible to reflect a solution of \eqref{P} into a solution of $(-\Delta)^su=u^p$ in the full space and arguments based on the method of moving-planes seem not suitable for such a situation.
 
 Our proof of Theorem \ref{non exitence result} is based on a new Pohozaev identity, in dimension $n=1$, which we state here below.

\begin{theorem}\label{Pohozaev}Let $u\in H^s_{\R_+}\cap L^\infty (\R_+)$ be weak solution of \eqref{P} with $n=1$, $0<s<\frac{1}{2}$ and $1<p<\frac{1+2s}{1-2s}$. 

Then $u$ fulfills the relation \[
    (1-2s)\frac{c_s}{2}\iint_{T(\R^+)}\frac{(u(x)-u(y))^2}{|x-y|^{1+2s}}dxdy=
                           -2\int_0^{+\infty}(-\Delta)^su(x)x\dot u(x)dx
    .\]
  \end{theorem}
 
 Usually, these kind of identities follow from multiplying the equation by $x\cdot\nabla u$ and integrating by parts. However, since we are working with weak solutions, this is equivalent to using $x\cdot \nabla u$ as a test function in the weak formulation. The problem is that, in general, if $u\in H^s_{\R^n_+}$ then $u$ is not differentiable up to $\partial \R^n_+$, furthermore it is not even clear whether $x\cdot \nabla u$ belongs to $H^s_{\R^n_+}.$ Hence, in order to prove Theorem \ref{Pohozaev}, we rely on an approximation argument. This argument is based on some decay estimates of $u$ and its gradient with respect to the distance from $\partial \R^n_+$. The precise result is the following: 
\begin{lemma}[see Lemma \ref{decay}]\label{decay-intro} Let $s\in(0,1)$, $n>2s$, and $1<p<\frac{n+2s}{n-2s}$.
Then, there exists a constant $C$, that depends on $n$, $s$, and $p$, such that for every weak solution $u\in H^s_{\R^n_+}\cap L^\infty(\R^n_+)$ of \eqref{P} it holds \[
u(x)+|\nabla u(x)|^{\frac{2s}{2s+p-1}}\leq {C}\,{x_n^{-\frac{2s}{p-1}}}\quad\text{for all $x\in\R^n_+.$}
\]
\end{lemma}

The proof of these decay estimates relies on the Doubling Lemma of Pol\'ačik, Quittner, and Souplet \cite[Lemma 5.1]{PQS}. If $\mathcal N_s u = 0$, we can express $u$ in $\R^n_-$ via an integral over $\R^n_+$. Combining this representation with the decay of $u$ in the positive half-space, we obtain similar estimates in $\R^n_-$ as well.
Observe that the above Lemma holds in any dimension $n$ but the decay estimate just involves the last coordinate $x_n$.

   Clearly it would be very interesting to prove the Pohozaev identity  in any dimension. The underlying idea is roughly the same, but the technical difficulty relies on the fact that we do not have a decay of $u$ with respect to the variables $x_1,..., x_{n-1}.$
\medskip

    The paper is organized as follows: in Section \ref{sec 2 lane emden}, we show the strict inequality $S_\mathcal{N}(\R^n_+)<S(\R^n)$, which implies the existence result in the critical case; Section \ref{sec 3 lane emden} contains the decay estimates of solutions and, in particular, the proof of Lemma \ref{decay-intro}; finally in Section \ref{sec 4 lane emden}, we prove the Pohozaev identity and the non-existence result in the subcritical regime.

\section{Existence of nontrivial solutions in the critical regime}\label{sec 2 lane emden}
As explained in the introduction, Theorem \ref{existence-critical} follows directly from the strict inequality between the Sobolev constants of Theorem \ref{strict inequality} and \cite[Theorem 1.3]{MN}.

Let us prove now Theorem \ref{strict inequality}.
\begin{proof}[Proof of Theorem \ref{strict inequality}]In the sequel, $c$ will denote a positive constant whose value may change from line to line; subscripts indicate the parameters it depends on.

    It is well known (see for instance \cite{LIEB}) that $S(\R^n)$ is attained by multiplications, dilations, and translations of the so--called \emph{Talenti Function}. Hence, if we define for every $t>0$ the function \[U_t(x)=c_{n,s} \Bigg(\frac{1}{1+\left|x-t e_n\right|^2}\Bigg)^{\frac{n-2s}{2}},\] where $c_{n,s}$ is a constant (independent of $t$) such that $\|U_t\|_{L^{2^*_s}(\R^n)}=1$, we have \[
    S(\R^n)=\E(U_t\:,\:\R^{2n}).
    \]Now we observe that \begin{align}\label{denominator}\nonumber\int_{\mathbb{R}_{+}^n}\left|U_t\right|^{\frac{2 n}{n-2 s}}dx&=1-c_{n,s}\int_{\mathbb{R}^n_-}\bigg(\frac{1}{1+\left|x-t e_n\right|^2}\bigg)^n d x\\\nonumber&=1-\frac{c_{n,s}}{t^{2n}}\int_{\mathbb{R}_{-}^n} \bigg(\frac{1}{\frac{1}{t^2}+\left|\frac{x}{t}-e_n\right|^2}\bigg)^n d x\\&=1-\frac{c_{n,s}}{t^n} \int_{\mathbb{R}^n_-}\bigg(\frac{1}{\frac{1}{t^2}+|x-e_n|^2}\bigg)^n d x.
    \end{align}Thanks to the {dominated convergence theorem, the integral} in \eqref{denominator} converges to $$\int_{\mathbb{R}^n_-}\frac{1}{|x-e_n|^{2n}} d x,$$ and hence we have \begin{equation}\label{denominator_asimptotic}
        \int_{\mathbb{R}_{+}^n}\left|U_t\right|^{\frac{2 n}{n-2 s}}=1-O\left(\frac{1}{t^n}\right)\quad\text{as $t\to+\infty$.}
    \end{equation}Now we observe that
    \begin{align}\nonumber\label{numerator}
        \E(U_t\:,\:(\R^n_-)^2)&=c_{n,s}\iint_{\left( \mathbb{R}_{-}^n\right)^2}\frac{\left(\left(\frac{1}{1+\left|x-t e_n\right|^2}\right)^{\frac{n-2s}{2}}-\left(\frac{1}{1+\left|y-t e_n\right|^2}\right)^{\frac{n-2s}{2}}\right)^2}{|x-y|^{n+2s}} d x d y\\
        &\nonumber=\frac{c_{n,s}}{t^{3n-2s}}\iint_{\left( \mathbb{R}_{-}^n\right)^2}\frac{\left(\left(\frac{1}{\frac{1}{t^2}+\left|\frac{x}{t}- e_n\right|^2}\right)^{\frac{n-2s}{2}}-\left(\frac{1}{\frac{1}{t^2}+\left|\frac{y}{t}- e_n\right|^2}\right)^{\frac{n-2s}{2}}\right)^2}{|\frac{x}{t}-\frac{y}{t}|^{n+2s}} d x d y\\
        &=\frac{c_{n,s}}{t^{n-2s}}\iint_{\left( \mathbb{R}_{-}^n\right)^2}\frac{\left(\left(\frac{1}{\frac{1}{t^2}+\left|x- e_n\right|^2}\right)^{\frac{n-2s}{2}}-\left(\frac{1}{\frac{1}{t^2}+\left|y- e_n\right|^2}\right)^{\frac{n-2s}{2}}\right)^2}{|x-y|^{n+2s}} d x d y.
    \end{align}We claim that the integral in \eqref{numerator} converges to \begin{equation}\label{limit integral}\iint_{\left( \mathbb{R}_{-}^n\right)^2}\frac{\left(\frac{1}{\left|x- e_n\right|^{n-2s}}-\frac{1}{\left|y- e_n\right|^{n-2s}}\right)^2}{|x-y|^{n+2s}} d x d y\end{equation} and that this last integral is finite. Given the claim, we have \[
    \E(U_t\:,\:(\R^n_-)^2)=O\left(\frac{1}{t^{n-2s}}\right)
    \] and combining this estimate with \eqref{denominator_asimptotic} we find \begin{align*}
    S_\mathcal N(\R^n_+)\leq\frac{\E(U_t\:,\:T(\R^{n}_+))}{\|U_t\|^2_{L^{2^*_s}(\R^n_+)}}=\frac{\E(U_t\:,\:\R^{2n})-\E(U_t\:,\:(\R^n_-)^2)}{\|U_t\|^2_{L^{2^*_s}(\R^n_+)}}=\frac{S(\R^n)-O(\frac{1}{t^{n-2s}})}{1-O(\frac{1}{t^n})}<S(\R^n)
    \end{align*} provided that $t$ is large enough. Hence we need only to show the claim. First we prove that \eqref{limit integral} is finite; in order to do this we observe that, arguing as in \cite[page 263]{LL}, we have \begin{equation}\label{LLestimate}
    \left|\frac{1}{|x-e_n|^{n-2 s}}-\frac{1}{|y-e_n|^{n-2 s}} \right|\leq c_{n,s}|x-y|^\alpha \max \left\{\frac{1}{|x-e _n|^{n-2 s+\alpha}},\frac{1}{| y-e_n|^{n-2 s+\alpha}}\right\}
    \end{equation}where $\alpha$ can be any number in $(0,1)$. Now we split the integral in \eqref{limit integral} as \[
    \iint_{\left( \mathbb{R}_{-}^n\right)^2}\frac{\left(\frac{1}{\left|x- e_n\right|^{n-2s}}-\frac{1}{\left|y- e_n\right|^{n-2s}}\right)^2}{|x-y|^{n+2s}} d x d y=I_1+I_2\]where\[I_1=\iint_{(\mathbb{R}_{-}^n)^2\cap\{|x-y|>\frac{1}{2}\}}\frac{\left(\frac{1}{\left|x- e_n\right|^{n-2s}}-\frac{1}{\left|y- e_n\right|^{n-2s}}\right)^2}{|x-y|^{n+2s}} d x d y,\]and\[I_2=\iint_{(\mathbb{R}_{-}^n)^2\cap\{|x-y|\leq\frac{1}{2}\}}\frac{\left(\frac{1}{\left|x- e_n\right|^{n-2s}}-\frac{1}{\left|y- e_n\right|^{n-2s}}\right)^2}{|x-y|^{n+2s}} d x d y.
    \]In order to estimate $I_1$ we use \eqref{LLestimate} with $\alpha_1<s$ and we get \begin{align}\nonumber\label{I1}
        I_1&\leq c_{n,s}\iint_{(\R^n_-)^2\cap\{|x-y|>\frac{1}{2}\}}\frac{\max\left\{\frac{1}{|x-e_n|^{n-2s+\alpha_1}},\frac{1}{|y-e_n|^{n-2s+\alpha_1}}\right\}^2}{|x-y|^{n+2s-2\alpha_1}}dxdy\\&\leq  c_{n,s}\int_{\mathbb{R}^n_-} \frac{dx}{|x-e_n|^{2 n-4 s+2 \alpha_1}} \int_{\{|x-y|>\frac
12\}} \frac{dy}{|x-y|^{n+2 s-2 \alpha_1}}.
    \end{align}
    If $\alpha_1<s$ the integral with respect to the variable $y$ in \eqref{I1} is convergent, and hence (up to relabelling the constant) \[
    I_1\leq c_{n,s}\int_{\mathbb{R}^n_-} \frac{dx}{|x-e_n|^{2 n-4 s+2 \alpha_1}},
    \]and this last integral is finite if and only if $1+n-4s+2\alpha_1>1$, which is true if we can choose $\alpha_1>2s-\frac{n}{2}$. We can find  $\alpha_1$ that fulfills both conditions if \[
    2s-\frac{n}{2}<s\iff2s<n,
    \]which is true by assumption. To estimate $I_2$ we argue in the same way but by choosing $s<\alpha_2<1$. In this way, we get \[
    I_2\leq c_{n,s}\int_{\mathbb{R}^n_-} \frac{dx}{|x-e_n|^{2 n-4 s+2 \alpha_2}} \int_{\{|x-y|\leq\frac
12\}} \frac{dy}{|x-y|^{n+2 s-2 \alpha_2}}.
    \]Since $\alpha_2>s$, the integral with respect to the variable $y$ is convergent and hence  $$I_2\leq c_{n,s}\int_{\mathbb{R}^n_-} \frac{dx}{|x-e_n|^{2 n-4 s+2 \alpha_2}},$$ which is again convergent since $n>2s$.
    
	Let \[
		\widetilde U _t(x)= \Bigg(\frac{1}{\frac{1}{t^2}+\left|\frac xt- e_n\right|^2}\Bigg)^{\frac{n-2s}{2}}
		,\]we observe that, arguing again as in \cite[page 263]{LL}, we get\begin{align*}
    \frac{\big(\widetilde U_t(x)-\widetilde U_t(y)\big)^2}{|x-y|^{n+2s}}&\leq c_{n,s}\frac{\left(|x-e_n|^2-|y-e_n|^2\right)^{2\alpha}}{|x-y|^{n+2s}}\max\left\{\frac{1}{|y-e_n|^{n-2s+2\alpha}},\frac{1}{|x-e_n|^{n-2s+2\alpha}}\right\}^2 \\&\colon=V(x,y).
    \end{align*}Hence, in order to apply the {dominated convergence theorem} to \eqref{numerator} it is enough to prove that $V$ is in $L^1((\R^n_-)^2)$ and, by symmetry, it is enough to show that $V$ belongs to $L^1(D)$, where $$D=(\R^n_-)^2\cap \{|y-e_n|<|x-e_n|\}.$$ If $(x,y)\in D$ we have \begin{align}\label{estimate}
   \nonumber V(x,y)&=c_{n,s}\frac{\left(|x-e_n|^2-|y-e_n|^2\right)^{2\alpha}}{|x-y|^{n+2s}}\frac{1}{|y-e_n|^{2n-4s+4\alpha}}\\&\leq  c_{n,s}\frac{\left(|y-e_n||x-y|+|x-y|^2\right)^{2\alpha}}{|x-y|^{n+2s}}\frac{1}{|y-e_n|^{2n-4s+4\alpha}}.
    \end{align}Now we observe that \begin{align*}
       \iint_{D\cap\{|x-y|<\frac{1}{2}\}}V(x,y)dxdy\leq c_{n,s,\alpha}\int_{\R^n_-}\frac{dy}{|y-e_n|^{2n-4s+2\alpha}}\int_{\{|x-y|<\frac{1}{2}\}}\frac{dx}{|x-y|^{n+2s-2\alpha}}.
    \end{align*}In order to have convergence of this integral it is enough to choose $\alpha:=\alpha_3\in(s,1).$ Splitting the integral over $D\cap\{|x-y|\geq\frac{1}{2}\}$ into the regions where $|y-e_n|\leq |x-y|$ and where $|y-e_n|\geq |x-y|$ and using estimate \eqref{estimate} twice we find \begin{equation}\label{last two integrals}\begin{aligned}
    \iint_{D\cap\{|x-y|\ge\frac{1}{2}\}}V(x,y)dxdy\leq &c_{n,s,\alpha_4}\int_{\R^n_-}\frac{dy}{|y-e_n|^{2n-4s+4\alpha_4}}\int_{\{|x-y|\ge\frac{1}{2}\}}\frac{dx}{|x-y|^{n+2s-4\alpha_4}}\\
    &+c_{n,s,\alpha_5}\int_{\R^n_-}\frac{dy}{|y-e_n|^{2n-4s+2\alpha_5}}\int_{\{|x-y|\geq\frac12\}}\frac{dx}{|x-y|^{n+2s-2\alpha_5}}.
    \end{aligned}\end{equation}If we choose a positive $\alpha_4\in (s-\frac{n}{4},\frac{s}{2})$ the first two integrals in \eqref{last two integrals} are convergent, furthermore if we choose a positive $\alpha_5\in(2s-\frac{n}{2},s)$ the other two integrals are also convergent (note that the two intervals are non-empty since $n>2s$). The above estimates show that $V\in L^1(D)$ and conclude the proof.
\end{proof}

We end this {section} with the proof of Theorem \ref{existence-critical}.
\begin{proof}[Proof of Theorem \ref{existence-critical}]
The result follows by combining Theorem \ref{strict inequality} with \cite[Theorem 1.3]{MN}.

\end{proof}
\section{Decay of solutions}\label{sec 3 lane emden}
In this section, we will need to consider pointwise, and not just weak, solutions to problem \eqref{P}. For this, we recall that the fractional Laplacian $(-\Delta)^s$ is well defined for functions {$u\in \mathcal L_s\cap C^{2s+\alpha}$} for some $\alpha >0$, where the space  $\mathcal L_s$ is defined as
\begin{equation}\label{L_s}
\mathcal L_s:=\left\{u:\R^n\to \R\,:\,\int_{\R^n}\frac{|u(x)|}{1+|x|^{n+2s}}\,dx <\infty\right\}.
\end{equation}

The following inclusion has been essentially shown in \cite[Lemma 2.3]{CC2020}.
\begin{lemma}\label{L-s}
We have:
$$ H^s_{\R^n_+} \subset \mathcal L_s.$$
\end{lemma}
\begin{proof}

%

We recall that if $u\in H^s_{\R^n_+}$, then it satisfies:

\begin{equation}\label{good}
    u\in L^{2^*_s}(\R^n_+) \quad \mbox{and}\quad \iint_{\R^{2n}\setminus(\R^n_-)^2}\frac{|u(x)-u(y)|^2}{|x-y|^{n+2s}}dx\,dy <\infty.
\end{equation}

The proof follows the lines of \cite[Lemma 2.3]{CC2020}. We report it here, for the reader's convenience.

We show that condition \eqref{good} implies
\begin{equation}\label{L_2}
\int_{\R^n}\frac{|u(x)|^2}{1+|x|^{n+2s}}\,dx<\infty,
\end{equation}
which, in particular, implies that $u \in \mathcal L_s$, by using the H\"older inequality and observing that $(1+|x|^{n+2s})^{-1}\in L^1(\R^n)$. 

Throughout this proof we denote by $C$ many different positive constants whose precise value is not important for the goal of the proof and may change from line to line.
Let $\Omega$ be a compact set contained in $\R^n_-$. We have
\begin{equation}\label{chain1}
\begin{split}
\infty &>\int_{\R^n_+}\int_{\R^n }\frac{|u(x)-u(y)|^2}{|x-y|^{n+2s}}\,dx\,dy  \\
& \ge\iint_{(\R^n_+)^2}\frac{|u(x)-u(y)|^2}{|x-y|^{n+2s}}\,dx\,dy + \int_{\Omega} \int_{\R^n_-}\frac{|u(x)-u(y)|^2}{|x-y|^{n+2s}}\,dx\,dy \\
& \ge \iint_{(\R^n_+)^2} \frac{|u(x)-u(y)|^2}{|x-y|^{n+2s}}\,dx\,dy \\
&\hspace{2em}+ \frac{1}{2}\int_{\Omega} \int_{\R^n_-}\frac{|u(x)|^2}{|x-y|^{n+2s}}\,dx\,dy -\int_{\Omega} \int_{\R^n_-}\frac{|u(y)|^2}{|x-y|^{n+2s}}\,dx\,dy,
\end{split}
\end{equation}
where, in the last estimate we used that $|a-b|^2 \ge \frac{1}{2} a^2-b^2$ by the Young inequality.

Since $u$ satisfies \eqref{good}, clearly the first term on the right-hand side is finite. Moreover, using that for $x\in \R^n_-$ and $y \in \Omega$ one has that $|x-y|\ge \omega$, for some $\omega>0$, and the integrability of the kernel at infinity, we have for every $y\in\Omega$
$$
\int_{\R^n_-}\frac{1}{|x-y|^{n+2s}}\,dx\le C\int_{\omega}^\infty \tau^{n-1-(n+2s)} d\tau=\frac{C}{\omega^{2s}},
$$
where $C$ is independent of $y\in\Omega$.
Hence, 

\begin{equation}\label{chain1ba}
\begin{split}
\int_{\Omega} \int_{\R^n_-}\frac{|u(y)|^2}{|x-y|^{n+2s}}\,dx\,dy&= \int_{\Omega} |u(y)|^2\bigg(\int_{\R^n_-}\frac{1}{|x-y|^{n+2s}}\,dx\bigg)\,dy\\
&\le  \frac{C}{\omega^{2s}}\int_{\Omega}|u(y)|^2 dy < \infty.
\end{split}
\end{equation}

Therefore, combining \eqref{chain1} with \eqref{chain1ba}, we deduce that
$$
\int_{\Omega} \int_{\R^n_-}\frac{ |u(x)|^2}{|x-y|^{n+2s}}\,dx\,dy < \infty.
$$

Finally, since $\Omega$ is bounded,  we have that there exists some number $d$ such that $|x-y|\le d+ |x|$ for every $x \in \R^n_-$ and $y \in \Omega$, which implies that
\begin{equation}\label{norm in R-}
\int_{\Omega} \int_{\R^n_-}\frac{ |u(x)|^2}{|x-y|^{n+2s}}\,dx\,dy\ge |\Omega| \int_{\R^n_-}\frac{ |u(x)|^2}{(d+|x|)^{n+2s}}\,dx.\end{equation}

Moreover, since $u\in L^{2^*_s}(\R^n_+)$ and $(d+|x|)^{-n-2s}\in L^1(\R^n_+)$, by using as before the H\"older inequality, we also have that 
$$\int_{\R^n_+}\frac{|u(x)|^{2}}{(d+|x|)^{n+2s}}dx<\infty,$$
which, together with \eqref{norm in R-}, gives the conclusion.

\end{proof}

We can now prove our decay estimates.

\begin{lemma}\label{decay}
Let $s\in(0,1)$, $n>2s$, and $1<p<\frac{n+2s}{n-2s}$. Assume that $u\in H^s_{\R^n_+}\cap L^\infty(\R^n_+)$ is a weak positive solution to
\begin{equation}\label{subcritical}\begin{cases}
(-\Delta)^s u=u^p &\text{in } \mathbb{R}_{+}^n \\
\mathcal{N}_s u=0 &\text{in } \mathbb{R}_{-}^n.
\end{cases}\end{equation}

Then, there exists a constant $C$ that depends on $n,\,s$ and $p$, such that 
\[
u(x)+|\nabla u(x)|^{\frac{2s}{2s+p-1}}\leq {C}\,{x_n^{-\frac{2s}{p-1}}}\quad\text{for all $x\in\R^n_+.$}
\]
\end{lemma}\begin{proof}Assume by contradiction that there exists a sequence of $u_k\in H^s_{\R^n_+}$ of solutions to \eqref{subcritical}, and a sequence of points $x^k\in \R^n_+$ such that\[
M_k(x^k)\geq\frac{2k}{x_n^k}
\]where\[
M_k:=u_k^{\frac{p-1}{2 s}}+|\nabla u_k|^{\frac{p-1}{2s+p-1}}.
\]By the Doubling Lemma \cite[Lemma 5.1]{PQS} there exists a sequence $y^k\in\R^n_+$ such that \[
M_k(y^k) y_n^k \geq 2 k,
\]and \[
M_k(z) \leq 2 M_k(y_k)^{}\quad\text{if $z\in B_{k\lambda_k}(y_k)$}
\]where $\lambda_k=M_k(y^k)^{-1};$ notice that $B_{k\lambda_k}(y_k)\subset\R^n_+$.
Now, for any $x\in \R^n,$ we define \[\widetilde{u}_k(x)=\lambda_k^{\frac{2s}{p-1}} u_k(y^k+\lambda_k x),\]and we notice that $\widetilde{u}_k$ solves \begin{equation}\label{equation in a ball}
(-\Delta)^s \widetilde{u}_k=\widetilde{u}_k^p \quad  \text{in } B_k(0);
\end{equation}moreover \begin{equation}\label{bound in 0 sequence}
\widetilde{u}_k(0)^{\frac{p-1}{2s}}+|\nabla \widetilde u_k(0)|^{\frac{p-1}{2s+p-1}}=1\end{equation} and \begin{equation}\label{bound in a ball}\widetilde{u}_k(x)^{\frac{p-1}{2s}}+|\nabla \widetilde u_k(x)|^{\frac{p-1}{2s+p-1}}\leq 2\quad\text{if $x\in B_k(0)$}.
\end{equation}
From \eqref{bound in a ball}, the Ascoli--Arzelà Theorem and \cite[Theorem 1.1 and Theorem 1.3]{chenLiWuXin} we have that, up to subsequences, $\widetilde u_k$ converges in {$C^{\max\{1,2s\}+\alpha}_{\text{loc}}(\R^n)$ } to some non-negative function $u\in C^1(\R^n).$ Observe that we crucially use that the regularity estimates of  \cite[Theorem 1.1 and Theorem 1.3]{chenLiWuXin} only depend on the $L^\infty$-norm of $u$ in $B_k(0)$ and not in the whole $\R^n$. Hence, passing to the limit in \eqref{bound in 0 sequence} and \eqref{bound in a ball} we find \begin{equation}\label{bound in 0 limit}
u(0)^{\frac{p-1}{2s}}+|\nabla  u(0)|^{\frac{p-1}{2s+p-1}}=1\end{equation} and \begin{equation}u(x)^{\frac{p-1}{2s}}+|\nabla  u(x)|^{\frac{p-1}{2s+p-1}}\leq 2\quad\text{if $x\in\R^n$}.\label{bound limit}
\end{equation}

We want now to pass to the limit in the equation \eqref{equation in a ball}. {The proof takes inspiration from \cite[Proof of Theorem 1.8.]{chenLiWuXin}.}
{Since $\widetilde u_k \in C^{2s + \alpha}\cap \mathcal L_s$ (see Lemma \ref{L-s}), we have that $\widetilde u_k$ solves equation \eqref{equation in a ball} pointwise and therefore $(-\Delta)^s \widetilde u_k$ converges pointwise in $\R^n$.} Hence from \cite[Theorem 1.1]{Du-Jin-Xiong-Yang} we have \[
\lim _{k \rightarrow \infty}(-\Delta)^s \widetilde u_k(x)=(-\Delta)^s u(x)-b \quad \text{for all } x\in \mathbb{R}^n,
\]where $b\in[0,+\infty)$ is defined as\begin{equation}\label{tail}
b=c_{n, s} \lim _{R \rightarrow \infty} \lim _{k \rightarrow \infty} \int_{B_R^c} \frac{\widetilde u_k(x)}{|x|^{n+2 s}} d x.
\end{equation}Therefore $u$ solves \begin{equation}\label{equation with b}
(-\Delta)^su(x)=u^p(x)+b\quad\text{for all $x\in\R^n.$ }
\end{equation}We want to show that $b=0.$ If, by contradiction, $b>0$ we define for any $R>0$ \[
T_R(x)=c_{n,s}\int_{\R^n}\frac{(u^p(y)+b)\chi_{B_R}(y)}{|x-y|^{n-2s}}dy,
\]
that is a solution of 
\[
\begin{cases}
    (-\Delta)^sT_R(x)&=(u^p(x)+b)\chi_{B_R}(x)\quad\text{$x\in\R^n$}\\
   {\displaystyle \lim_{|x|\to+\infty}T_R(x)}&=0,
\end{cases}
\]
where $\chi_{B_R}$ denotes the characteristic function of the ball $B_R$ and $1/|x-y|^{n-2s}$ is the fundamental solution of the $s$-Laplacian on $\R^n$.

Now we define $u_R(x)=u(x)-T_R(x)$ and we see that \[
\begin{cases}
(-\Delta)^s u_R(x) &\geq 0 \quad x \in \mathbb{R}^n, \\
{\displaystyle\lim _{|x| \rightarrow +\infty} u_R(x)} &\geq 0.
\end{cases}
\]
Hence, from the maximum principle of \cite[Theorem 1]{CHEN2018735} we have that $u_R$ is non--negative. From this and the Monotone Convergence Theorem we obtain \[
u(x)\geq \lim_{R\to+\infty}T_R(x)\geq c_{n,s}\int_{\R^n}\frac{u^p(y)+b}{|x-y|^{n-2s}}dy\geq c_{n,s}\int_{\R^n}\frac{b}{|x-y|^{n-2s}}dy,
\] for any $x\in\R^n$. If $b>0$ the last integral on the right hand side is divergent, which is a contradiction and hence $b=0.$ From \eqref{equation with b} we find that $u$ is a non--negative solution of \[
(-\Delta)^su(x)=u^p(x)\quad\text{for all $x\in\R^n,$}
\]and hence, from \cite[Theorem 1.8 and Remark 1.9]{JLX}, $u$ must vanish everywhere, which is not possible thanks to \eqref{bound in 0 limit}.
\end{proof}
Given a point $x\in\R^n$ we write $x=(x',x_n)$ with $x'\in\R^{n-1}$ and $x_n\in\R.$
\begin{lemma}\label{doubling-derivative}
  Let $s\in(0,1)$, $n>2s$, and $1<p<\frac{n+2s}{n-2s}$. Assume that $u\in H^s_{\R^n_+}\cap L^\infty(\R^n_+)$ is a weak positive solution to \eqref{subcritical}.

Then, there exists a constant $C_2$ that depends on $n,\,s$ and $p$, such that
\[
|\nabla u(x)|\leq \frac{C_2}{x_n}\quad\text{for any $x\in\R^n_+$ with $0<x_n\leq 1.$}
\]
\end{lemma}\begin{proof}
    The proof again relies on the Doubling Lemma of \cite[Lemma 5.1]{PQS}. Indeed assume by contradiction the existence of a sequence $y^k\in \R^n_+$ with $y_n^k\leq 1$ such that \[
    |\nabla u(y_k)|\geq \frac{2k}{y_n^k}.
    \]By the Doubling Lemma there exists a sequence $x^k\in\R^n_+$ with $x_n^k\leq 1$ such that \[
    |\nabla u(x^k)|\geq \frac{2k}{x^k_n}\text{,}\quad |\nabla u(x^k)|\geq|\nabla u(y^k)|
    \]and \[
    |\nabla u(z)|\leq2|\nabla u(x^k)|\quad \text{for all $z\in B(x^k,k\lambda_k)$,}
    \]where $\lambda_k=[|\nabla u(x^k)|]^{-1}.$ Notice that, since $x_n^k\in (0,1]$, $\lambda_k\to 0.$ Now we define, for every $y\in B(0,k)$ the function \[
    v_k(y)=u(x_k+\lambda_ky)
    \]and we notice that \[
    |\nabla v_k(y)|\leq 2\quad \text{,}\quad  |\nabla v_k(0)|=1
    \]and \[
    (-\Delta)^sv_k(y)=\lambda_k^{2s}v_k^p(y)\quad\text{in $B(0,k)$}.
    \]Arguing as in the proof of Lemma \ref{decay} we find that $v_k$ converges in $C^{\max\{1,2s\}+\alpha}_{\text{loc}}(\R^n)$ to some function $v\in C^1(\R^n).$ Moreover the limit function solves \[
    (-\Delta)^sv(x)=b\quad\text{for all $x\in\R^n,$}
    \]where $b$ is defined analogously to \eqref{tail} with $v_k$ in place of $\widetilde u_k.$ Arguing again as in the proof of Lemma \ref{decay} we can see that $b=0$, hence $v$ is a bounded solution of\[
    (-\Delta)^sv(x)=0\quad\text{for all $x\in\R^n,$}
    \]therefore $v$ must be constant (see for instance \cite[Proposition 2]{ZuoChenCuiYuan}). But this is not possible since $$|\nabla v(0)|=\lim_{k\to\infty}|\nabla v_k(0)|=1.$$
    \end{proof}

The nonlocal Neumann condition $\mathcal N_s u =0$ in $\R^{n}_-$, allows us to deduce decay estimates for $u$ in the negative half-space, from the ones obtained in the positive half-space.
    
\begin{corollary}\label{decay2}
 Let $s\in(0,1)$, $n>2s$, and $1<p<\frac{n+2s}{n-2s}$. Assume that $u\in H^s_{\R^n_+}\cap L^\infty(\R^n_+)$ is a weak positive solution to \eqref{subcritical} and let \[
    \beta =\min\bigg\{1,\frac{2s}{p-1}\bigg\}.
    \]
    
    Then, there exists a constant $C_1>0$ such that \begin{equation}\label{decay function R^n-}
    u(x)\leq C_1\min\left\{1,|x_n|^{- \beta}\right\}\quad\text{for any $x\in\R^n_-$},
    \end{equation} and \begin{equation}\label{decay gradient R^n-}
    |\nabla u(x)|\leq C_1\frac{\min\left\{1,|x_n|^{- \beta}\right\}}{|x_n|}\quad\text{for any $x\in\R^n_-$}.
    \end{equation}
\end{corollary}\begin{proof}In the following $C_{n,s}$ will denote possibly different positive constants which depend only on $n$ and $s$.
First of all we notice that, if $x\in\R^n_-$, thanks to some algebraic manipulation and a change of variables, we have \begin{equation}\begin{aligned}\label{computation cap 3}
\int_{\R^n_+}\frac{dy}{|x-y|^{n+2s}}&=\int_{\R^n_+}\frac{dy}{\left(|x'-y'|^2+(x_n-y_n)^2\right)^\frac{n+2s}2}\\&=\int_{0}^{+\infty}\frac{dy_n}{|x_n-y_n|^{n+2s}}\int_{\R^{n-1}}\frac{dy'}{\left(|\frac{x'-y'}{x_n-y_n}|^2+1\right)^{\frac{n+2s}{2}}}\\&=\int_{0}^{+\infty}\frac{dy_n}{|x_n-y_n|^{1+2s}}\int_{\R^{n-1}}\frac{dz'}{\left(|z'|^2+1\right)^{\frac{n+2s}{2}}}\\&=C_{n,s}|x_n|^{-2s}.
\end{aligned}\end{equation}

We recall that, since $\mathcal N_su =0$, one has that
$$
 u(x)=\frac{\displaystyle\int_{\R^n_+}\frac{u(y)}{|x-y|^{n+2s}}dy}{\displaystyle\int_{\R^n_+}\frac{dy}{|x-y|^{n+2s}}},\quad \mbox{for every } x \in \mathbb R^n_-.
$$
    Hence, if $u$ is a solution of \eqref{subcritical} and $x\in\R^n_-$, we have \begin{equation}\label{first integral 1d}\begin{aligned}
    |u(x)|&\leq C_{n,s}|x_n|^{2s}\int_{\R^n_+}\frac{|u(y)|}{|x-y|^{n+2s}}dy\\&=C_{n,s}|x_n|^{2s}\bigg(\int_{\{0<y_n<1\}}\frac{|u(y)|}{|x-y|^{n+2s}}dy+\int_{\{y_n\geq1\}}\frac{|u(y)|}{|x-y|^{n+2s}}dy\bigg).\end{aligned}\end{equation}
    Since $u$ is bounded we can estimate the first integral arguing as before; this time we obtain 
    \begin{align}\label{first integral estimate}
     \nonumber    \int_{\{0<y_n<1\}}\frac{|u(y)|}{|x-y|^{n+2s}}dy&\leq\|u\|_{L^\infty}\int_{\R^{n-1}}dy'\int_0^{1}\frac{dy_n}{\left(|x'-y'|^2+(x_n-y_n)^2\right)^\frac{n+2s}{2}}
      \nonumber  \\&=\|u\|_{L^\infty}\int_{\R^{n-1}}\frac{dz'}{\left(|z'|^2+1\right)^{\frac{n+2s}{2}}}\int_{0}^1\frac{1}{|\nonumber x_n-y_n|^{1+2s}}dy_n
     \nonumber    \\&=\|u\|_{L^\infty}C_{n,s}\left(\frac{1}{|x_n|^{2s}}-\frac{1}{(1+|x_n|)^{2s}}\right)
      \nonumber   \\&\leq \frac{\|u\|_{L^\infty}C_{n,s}}{|x_n|^{1+2s}}.
    \end{align}
    For the second integral in \eqref{first integral 1d} we use Lemma \ref{decay} to infer \begin{align}\label{integral y>1}
      \nonumber  \int_{\{y_n\geq1\}}\frac{|u(y)|}{|x-y|^{n+2s}}dy&\leq\int_{\{y_n\geq1\}}\frac{1}{y_n^{\frac{2s}{p-1}}|x-y|^{n+2s}}dy
    \nonumber    \\&=\int_1^{+\infty}{y_n^{-\frac{2s}{p-1}}}dy_n\int_{\R^{n-1}}\frac{dy'}{\left(|x'-y'|^2+(x_n-y_n)^2\right)^\frac{n+2s}{2}}
     \nonumber   \\&=C_{n,s}\int_{1}^{+\infty}\frac{1}{y_n^\frac{2s}{p-1}|x_n-y_n|^{1+2s}}dy_n.
    \end{align}
    Performing the change of variables $y_n=-wx_n$ in the last integral we find \begin{equation}\label{integral_con0}
   |x_n|^{2s} \int_{1}^{+\infty}\frac{1}{y_n^\frac{2s}{p-1}|x_n-y_n|^{1+2s}}dy_n=\frac{1}{|x_n|^{\frac{2s}{p-1}}}\int_{\frac{1}{|x_n|}}^{+\infty}\frac{1}{w^{\frac{2s}{p-1}}|1+w|^{1+2s}}dw.
    \end{equation}
    In order to estimate this last integral, we fix $M>1$ (here we are interested in the decay of $u$ for $|x_n|$ large, i.e. $1/|x_n|$ small) and we split the integral as \begin{align}\label{integral_cont}
   \nonumber \int_{\frac{1}{|x_n|}}^{+\infty}\frac{1}{w^{\frac{2s}{p-1}}|1+w|^{1+2s}}dz&=\int_{\frac{1}{|x_n|}}^{M}\frac{1}{w^{\frac{2s}{p-1}}|1+w|^{1+2s}}dw+\int_{M}^{+\infty}\frac{1}{w^{\frac{2s}{p-1}}|1+w|^{1+2s}}dw\\&=\int_{\frac{1}{|x_n|}}^{M}\frac{1}{w^{\frac{2s}{p-1}}|1+w|^{1+2s}}dw+C_{n,s},
   \end{align}where the last equality follows since the integral on $[M,+\infty)$ is convergent. For the second integral, 
    we have
    \[
   \int_{\frac{1}{|x_n|}}^{M}\frac{1}{w^{\frac{2s}{p-1}}|1+w|^{1+2s}}dw\le \int_{\frac{1}{|x_n|}}^{M}\frac{dw}{w^{\frac{2s}{p-1}}} \le  C_{n,s}\big(|x_n|^{\frac{2s}{p-1}-1}+1\big).
   \]
   
   Combining this last estimate with \eqref{integral_cont}, \eqref{integral_con0} and \eqref{integral y>1} we find \[
   |x_n|^{2s} \int_{\{y_n\geq1\}}\frac{|u(y)|}{|x-y|^{n+2s}}dy\leq \frac{C_{n,s}}{|x_n|^{\min\left\{\frac{2s}{p-1},1\right\}}}
   .\]This last inequality and \eqref{first integral estimate} imply \begin{equation}\label{first decay}
   u(x)\leq C_1|x_n|^{-\beta}\quad\text{for $x\in\R^n_-.$}
   \end{equation}From equations \eqref{computation cap 3}, \eqref{first integral 1d} and the fact that $u$ is bounded in $\R^n_+$ we easily get \[
   |u(x)|\leq C_1\quad \text{if $x\in\R^n_-$},
   \]which together with \eqref{first decay}, implies \eqref{decay function R^n-}.
   For the estimate in \eqref{decay gradient R^n-} it is enough to observe that, if $x\in\R^n_-$, \[
   |\nabla u(x)|\leq C_{n,s}\bigg(|x_n|^{2s-1}\int_{\R^n_+}\frac{|u(y)|}{|x-y|^{n+2s}}dy+|x_n|^{2s}\int_{\R^n_+}\frac{|u(y)|}{|x-y|^{n+2s+1}}dy\bigg)
   \]and performing the very same computation as before we reach the conclusion.
\end{proof}
\section{Pohozaev Identity and Liouville-type result}\label{sec 4 lane emden}
In the following we will assume $0<s<\frac{1}{2}$ and $n=1$. Given $u\in C^1{(\R)}$, we denote by $\dot u$ its derivative.

 In order to prove our Pohozaev identity, we will crucially use the decay estimates established in the previous section. 
Indeed, in Lemma \ref{energy expansion-v2} below, we will show that a Pohozaev identity holds true for functions which satisfy certain decay conditions close to the origin and at $\infty$, and such decays are satisfied by solutions to our problem \eqref{P} thanks to the results of Section \ref{sec 3 lane emden}.

In order to prove Lemma \ref{energy expansion-v2}, we will first prove the same result under stronger assumptions on the function $u$.

We start by establishing the result for functions that vanish both close to the origin and at $\infty$.

\begin{lemma}\label{energy expansion non local}
Let $u\in C^1_{\mathrm c}(\R\setminus\{0\})$. Then,
      \[
    (1-2s)\frac{c_s}{2}\iint_{T(\R^+)}\frac{(u(x)-u(y))^2}{|x-y|^{1+2s}}dxdy=-2\int_0^{+\infty}(-\Delta)^su(x)x\dot u(x)dx-2\int_{-\infty}^0\mathcal N_su(x)x\dot u(x)dx.
    \]
\end{lemma}\begin{proof}First of all we notice that
    \begin{align}\label{integral_expansion}
\nonumber(1-2s) \frac{c_{s}}{2} \iint_{T\left(\mathbb{R}^{+}\right)} \frac{(u(x)-u(y))^2}{|x-y|^{1+2 s}}&=
\frac{c_{s}}{2} \iint_{T\left(\mathbb{R}^{+}\right)}(u(x)-u(y))^2\left(\frac{2}{|x-y|^{1+2 s}}-(1+2 s) \frac{(x-y)^2}{|x-y|^{3+2 s}}\right)\\
&=c_{s} \iint_{T\left(\mathbb{R}^{+}\right)}(u(x)-u(y))^2\left(\frac{1}{|x-y|^{1+2 s}}-(1+2s) \frac{(x-y) \cdot x}{|x-y|^{3+2 s}}\right)\\
\nonumber&=c_s \iint_{T\left(\mathbb{R}^{+}\right)}(u(x)-u(y))^2 \frac{\partial}{\partial x}\left(\frac{x}{|x-y|^{1+2 s}}\right) ,
    \end{align}where in \eqref{integral_expansion} we used the symmetry of the integral. Now we split this last integral as a sum of three parts, namely \[
    c_s \iint_{T\left(\mathbb{R}^{+}\right)}(u(x)-u(y))^2 \frac{\partial}{\partial x}\left(\frac{x}{|x-y|^{1+2 s}}\right)=A+B+C
    \]where \[
    A=c_{ s} \int_0^{+\infty} \int_0^{+\infty}(u(x)-u(y))^2 \frac{\partial}{\partial x}\left(\frac{x}{|x-y|^{1+2 s}}\right)dxdy,
    \]
    \[
    B=c_{ s} \int_0^{+\infty} \int_{-\infty}^0(u(x)-u(y))^2 \frac{\partial}{\partial x}\left(\frac{x}{|x-y|^{1+2 s}}\right)dx dy
    ,\]and \[
    C=c_{ s} \int_{-\infty}^0 \int_0^{+\infty}(u(x)-u(y))^2 \frac{\partial}{\partial x}\left(\frac{x}{|x-y|^{1+2 s}}\right) dxdy.
    \]For $B$ and $C$ we find, after integration by parts with respect to the $x$ variable, that \begin{equation}\label{B}
    B=-2 c _s \int_0^{+\infty}  \int_{-\infty}^0 \frac{(u(x)-u(y))}{|x-y|^{ 1+2 s}} x \dot u(x)dxdy
    \end{equation}and\begin{equation}\label{C}
    C=-2 c _s \int_{-\infty}^0\int_0^{+\infty}  \frac{(u(x)-u(y))}{|x-y|^{ 1+2 s}} x \dot u(x)dxdy.
    \end{equation}For $A$ we have \begin{align*}
    A=\lim _{\delta  \to 0^+} &c_{ s} \int_0^{+\infty}  \int_{0}^{y-\delta}(u(x)-u(y))^2 \frac{\partial}{\partial x}\left(\frac{x}{|x-y|^{1+2 s}}\right) d xdy\\+& c_s \int_0^{+\infty}  \int_{y+\delta}^{+\infty}(u(x)-u(y))^2 \frac{\partial}{\partial x}\left(\frac{x}{|x-y|^{1+2s}}\right) d xdy.
    \end{align*}Again using the integration by parts we find \begin{align}\label{boundary1}
   \nonumber \int_0^{+\infty}  \int_{0}^{y-\delta}(u(x)-u(y))^2 \frac{\partial}{\partial x}\bigg(\frac{x}{|x-y|^{1+2 s}}\bigg) d xdy=&-2  \int_0^{+\infty} \int_{0}^{y-\delta} \frac{u(x)-u(y)}{|x-y|^{1+2s}} \dot{u}(x) xdxdy\\&+ \int_0^{+\infty}(u(y-\delta)-u(y))^2 \frac{y-\delta}{\delta^{1+2 s}} d y,
    \end{align}and\begin{align}\label{boundary2}
     \nonumber   \int_0^{+\infty}  \int_{y+\delta}^{+\infty}(u(x)-u(y))^2 \frac{\partial}{\partial x}\bigg(\frac{x}{|x-y|^{1+ 2s}}\bigg) d xdy=&-2  \int_0^{+\infty} \int_{y+\delta}^{+\infty} \frac{u(x)-u(y)}{|x-y|^{1+s s}} \dot{u}(x) xdxdy\\&-\int_0^{+\infty}(u(y+\delta)-u(y))^2 \frac{y+\delta}{\delta^{1+2 \delta}}dy.
    \end{align}
    Now we want to show that the sum of the single integrals in \eqref{boundary1} and \eqref{boundary2} goes to $0$ as $\delta\to0$. Indeed by a change of variables we can see that this sum is equal to \[
    \int_{-\delta}^0(u(y+\delta)-u(y))^2 \frac{y}{\delta^{1+2 s}}dy- \int_0^{+\infty} \frac{(u(y+\delta)-u(y))^2}{\delta^{2 s}}dy
    .\]For the first integral we notice that \[
    \int_{-\delta}^0(u(y+\delta)-u(y))^2 \frac{|y|}{\delta^{1+2 s}}dy\leq{2\|u\|^2_{L^\infty(\R)}}\delta^{1-2s}\rightarrow0
    \]as $\delta\to0$ since $s<\frac{1}{2}$. For the second integral we observe that, since the support of $u$ is compact, there exist $0<a<b$ such that \begin{equation}\label{boundary_1to0}
    \operatorname{supp}(u(\cdot+\delta)-u(\cdot))\subseteq[a,b]
    \end{equation}for any $\delta$ small enough. Therefore \begin{equation}\label{boundary_2to0}
    \int_0^{+\infty} \frac{(u(y+\delta)-u(y))^2}{\delta^{2 s}}dy\leq \int_{a}^b\frac{(u(y+\delta)-u(y))^2}{\delta^{2 s}}dy\leq \delta^{2-2s}(b-a)\|\dot  u\|_{L^\infty(\R^+)}\rightarrow 0
    \end{equation}as $\delta\to0.$
    From \eqref{boundary_1to0} and \eqref{boundary_2to0} we have \[
    A=-2c_s\int_0^{+\infty}\int_{0}^{+\infty}\dfrac{u(x)-u(y)}{|x-y|^{1+2s}}\dot u(x)xdxdy
    \]and combining this last relation with \eqref{B} and \eqref{C} and recalling \eqref{integral_expansion} we find\begin{align*}
    (1-2s)\frac{c_s}{2}\iint_{T(\R^+)}\frac{(u(x)-u(y))^2}{|x-y|^{1+2s}}dxdy=&-2\int_0^{+\infty}(-\Delta)^su(x)x\dot u(x)dx\\&-2\int_{-\infty}^0\mathcal N_su(x)x\dot u(x)dx,
    \end{align*}which concludes the proof.
    \end{proof}

    The following Leibniz-type rule holds true.
    \begin{lemma}
        Let $u,v:\R^n\to\R$ be two functions. Then,\begin{equation}\label{product rule laplacian}
        (-\Delta)^s(uv)(x)=(-\Delta)^su(x)v(x)+u(x)(-\Delta)^sv(x)-I(u,v,\R^n)(x)
        \end{equation}and\begin{equation}\label{product rule neumann}
        \mathcal{N}_s(uv)(x)=\mathcal{N}_su(x)v(x)+u(x)\mathcal{N}_sv(x)-I(u,v,\R^n_+)(x),
        \end{equation}where, for any $E\subseteq\R^n$ measurable, we denoted \[
        I(u,v,E)(x)=c_{n,s}\int_{E} \frac{(u(x)-u(y))\left(v(x)-v(y)\right)}{|x-y|^{n+2 s}}dy
       .\]
    \end{lemma}
    \begin{proof}
        The proof of \eqref{product rule laplacian} can be found in \cite[Proposition 1.5]{BWZ}. With the very same computation we also obtain \eqref{product rule neumann}.
    \end{proof}
    \begin{remark}\label{cut off}
    Let $\varphi:\R\to\R$ be a function with the following properties:\begin{enumerate}[$(i)$]
        \item $\varphi \in C^{\infty}(\R);$
            \item $0\leq\varphi (x)\leq1$ for all $x\in\R;$
            \item $\varphi(x)=1$ if $x\in\left(-\infty,-2\right]\cup [2,\infty);$
            \item $\varphi(x)=0$ if $x\in\left[-1,1 \right]$.
            \item $|\dot{\varphi}(x)|\le 1$ for all $x\in\R$.
    \end{enumerate}
    Then there exists a constant $C>0$ such that \[
    \int _{\R}\frac{|\varphi(x)-\varphi(y)|}{|x-y|^{1+2s}}dy\leq \frac{C}{(1+|x|)^{1+2s}}\quad\text{for all $x\in \R$.}
    \]Indeed if $|x|\in[0,3]$ we split
\[\begin{aligned}
\int_\R\frac{|\varphi(x)-\varphi(y)|}{|x-y|^{1+s}}dy&=\int_{\{|x-y|\geq1\}}\frac{|\varphi(x)-\varphi(y)|}{|x-y|^{1+2s}}dy+\int_{\{|x-y|\leq1\}}\frac{|\varphi(x)-\varphi(y)|}{|x-y|^{1+2s}}dy\\
&\leq \int_{\{|x-y|\geq1\}}\frac{2}{|x-y|^{1+2s}}dy+\int_{\{|x-y|\leq1\}}\frac{1}{|x-y|^{2s}}dy\\
&\leq \frac{2}{s}+\frac{L}{1-2s}\\&
\leq \frac{C}{(1+|x|)^{1+2s}}\quad \text{if $|x|\in[0,3]$.}
\end{aligned}
\]If $|x|\geq3$ we have $\varphi(x)=1$ and since also $\varphi(y)=1$ if $|y|\geq 2$ we can write \[
\int_{\R}\frac{|\varphi(x)-\varphi(y)|}{|x-y|^{1+2s}}dy=\int_{-2}^2\frac{|1-\varphi(y)|}{|x-y|^{1+2s}}dy,
\]and observing that for $|x|>3$ and $|y|<2$, we have
\[
|x-y|\geq |x|-|y|\geq |x|-2
\]we conclude \[
\int_\R\frac{|\varphi(x)-\varphi(y)|}{|x-y|^{1+2s}}dy\leq \frac{1}{(|x|-2)^{1+2s}}\int_{-2}^2{|1-\varphi(y)|}dy\leq \frac{C}{(1+|x|)^{1+2s}}.
\]if $|x|\in[3,+\infty).$

 This implies in particular  that 
 $$(-\Delta)^s\varphi\in L^1(\R)\cap L^\infty(\R_+),\quad \mathcal N_s\varphi\in L^1(\R)\cap L^\infty(\R_-).
 $$

\end{remark}
\medskip

In the following Lemma, we extend Lemma \ref{energy expansion non local} to functions that have compact support but not necessarily vanish around the origin. For this, we need an integrability condition on their fractional Laplacians and nonlocal Neumann data and a suitable bound on $\dot u$.

\begin{lemma}\label{energy expansion-v1}
  Let $u \in C^1_{\mathrm{loc}}(\R \setminus \{0\}) \cap L^\infty(\R)$ be a function with $u \equiv 0$ on $\R \setminus [-K,K]$ for some $K>0$ and with the following properties:
  \begin{itemize}
   \item[$(i)$] $(-\Delta)^s u \in L^1(0,\infty)$
   \item[$(ii)$] $\mathcal N_s u \in L^1(-\infty,0)$.
   \item[$(iii)$] $|\dot u(x)| \le \frac{C}{x} \;\text{for $x \in \R \setminus \{0\}$.}$
   \end{itemize}
  Then \[
    (1-2s)\frac{c_s}{2}\iint_{T(\R^+)}\frac{(u(x)-u(y))^2}{|x-y|^{1+2s}}dxdy=-2\int_0^{+\infty}(-\Delta)^su(x)x\dot u(x)dx-2\int_{-\infty}^0\mathcal N_su(x)x\dot u(x)dx.
    \]
\end{lemma}
  \begin{proof}
       Let $\varphi$ be a function as in Remark \ref{cut off}. We apply Lemma \ref{energy expansion non local} to the function $u\varphi_k$ where $\varphi_k(x) = \varphi( k x).$
    We then obtain:
      \begin{align*}
        &(1-2 s) \frac{c_s}{2} \int_{T\left(\mathbb{R}^{+}\right)} \frac{\left(u \varphi_k(x)-u \varphi_k(y)\right)^2}{|x-y|^{1+2 s}} d x d y\\
        &=-2 \int_0^{+\infty}x(-\Delta)^s\left(u \varphi_k\right) \dot{\left(u \varphi_k\right)}dx -2  \int_{-\infty}^0x\mathcal N_s\left(u \varphi_k\right) \dot{\left(u \varphi_k\right)}dx\\
&=-2 \int_{\frac1k}^{K}x(-\Delta)^s\left(u \varphi_k\right) \dot{\left(u \varphi_k\right)}dx -2  \int_{-\infty}^{-\frac1k}x\mathcal N_s\left(u \varphi_k\right) \dot{\left(u \varphi_k\right)}dx.          
      \end{align*}
  We can rewrite the first integral on the right-hand side as \begin{align*}
&-2 \int_{\frac1k}^{K}(-\Delta)^s\left(u \varphi_k\right) x\dot{(u \varphi_k)}dx\\
&=-2 \int_{\frac1k}^{K} x\varphi_k(-\Delta)^s u \dot{(u \varphi_k)}-2 \int_{\frac 1k}^{K} x\dot{(u {\varphi}_k)}\left[u(-\Delta)^s \varphi_k-I(u, \varphi_k,\mathbb{R})(x)\right]dx.
\end{align*}For the first integral on the right hand side we have
\begin{align*}
-2 \int_{\frac1k}^{K} x\varphi_k(-\Delta)^s u \dot{(u \varphi_k)}
&=-2\int_{\frac1k}^{K}x\varphi_k^2(-\Delta)^su\cdot\dot u\:dx-2\int_{\frac1k}^{K}x \varphi_k\dot \varphi_k(-\Delta)^su \cdot u\:dx\\
 &=-2\int_{\frac1k}^{K}x\varphi_k^2(-\Delta)^su\cdot\dot u\:dx-2\int_{\frac1k}^{\frac2k} x \varphi_k\dot \varphi_k(-\Delta)^su \cdot u\:dx.
\end{align*}
    By Lebesgue's theorem,
\[
  \int_{\frac1k}^{K}x\varphi_k^2(-\Delta)^su\cdot\dot u\:dx \to 
  \int_{0}^{K}(-\Delta)^su   x \cdot  \dot u\:dx = \int_{0}^{\infty}(-\Delta)^su   x \cdot  \dot u\:dx 
\]
as $k \to +\infty$.

 Moreover, since $|\dot \varphi_k|\le C k$ on $[\frac1k,\frac2k]$ with a constant $C>0$, we have
\[
  \Bigl|\int_{\frac1k}^{\frac2k} x \varphi_k\dot \varphi_k(-\Delta)^su \cdot u\:dx\Bigr| \le \int_{\frac1k}^{\frac2k} u  (-\Delta)^s u \:dx\Bigr| \to 0
\]
as $k \to +\infty$ again by Lebesgue's theorem and since $(-\Delta)^s u\in L^1(\R^+)$ and $u$ is bounded on $\R^+$.

Hence 
\[
\int_{\frac1k}^{K} x\varphi_k(-\Delta)^s u \dot{(u \varphi_k)} \to   
\int_{0}^{\infty}(-\Delta)^su   x \cdot  \dot u\:dx
\]
as $k \to +\infty$.

Next we consider the integral
\begin{align*}
&\int_{\frac 1k}^{K} x\dot{(u {\varphi}_k)}\left[u(-\Delta)^s \varphi_k-I(u, \varphi_k,\mathbb{R})(x)\right]dx\\
&= \int_{\frac 1k}^{K}\varphi_k  x\dot{u} \left[u(-\Delta)^s \varphi_k-I(u, \varphi_k,\mathbb{R})(x)\right]dx + \int_{\frac 1k}^{K}\dot \varphi_k  x u \left[u(-\Delta)^s \varphi_k-I(u, \varphi_k,\mathbb{R})(x)\right]dx.
\end{align*}
We now use the facts that
$$
(-\Delta)^s \varphi_k  = k^{2s} [(-\Delta)^s \varphi](k \cdot)
$$
and
$$
I(u, \varphi_k,\mathbb{R})(x) = k^{2s} I(u(\frac{\cdot}{k}),\varphi,\mathbb{R})(k \cdot)
$$
to see that, by a change of variables,
\begin{equation}\label{first laplacian}\begin{aligned}
  &\Bigl|\int_{\frac 1k}^{K}\varphi_k  x\dot{u} \left[u(-\Delta)^s \varphi_k-I(u, \varphi_k,\mathbb{R})(x)\right]dx\Bigr|\\
  &=k^{2s-2} \Bigl|\int_{1}^{k K}  \varphi  x  \dot{u}(\frac{\cdot}{k}) \left[u(\frac{\cdot}{k})  (-\Delta)^s \varphi -I(u(\frac{\cdot}{k}), \varphi,\mathbb{R})(x)\right]dx\Bigr|\\
  &\le Ck^{2s-1} \int_{1}^{k K}  \varphi  \Bigl| u(\frac{\cdot}{k})  (-\Delta)^s \varphi -I(u(\frac{\cdot}{k}), \varphi,\mathbb{R})(x)\Bigr|dx.
\end{aligned}\end{equation}

By Remark \ref{cut off}, we have that
$$
\int_{1}^{k K}  \varphi  \Bigl| u(\frac{\cdot}{k})  (-\Delta)^s \varphi\Bigr| dx
\le C\int_{1}^{\infty}  \varphi  \Bigl|(-\Delta)^s \varphi\Bigr| dx <\infty.
$$
Moreover, by Remark \ref{cut off}, we have
$$
\Bigl| I(u(\frac{\cdot}{k}), \varphi,\mathbb{R})(x)\Bigr|\le 2\|u\|_{L^\infty(\R)} \int_{\R} \frac{|\varphi(x)-\varphi(y)|}{|x-y|^{1+2s}}dy
\le \frac{C}{(1+x)^{1+2s}}
$$
for $x \in (0,\infty)$ and therefore
$$
\int_{1}^{k K}  \varphi \Bigl| I(u(\frac{\cdot}{k}), \varphi,\mathbb{R})(x)\Bigr|dx \le C \int_{1}^{\infty}\frac{1}{(1+x)^{1+2s}}dx < \infty.
$$
Since $s < \frac{1}{2}$, we thus conclude that 
$$
\int_{\frac 1k}^{K}\varphi_k  x\dot{u} \left[u(-\Delta)^s \varphi_k-I(u, \varphi_k,\mathbb{R})(x)\right]dx \to 0
$$
as $k \to +\infty$. Moreover, we have, again by a change of variables, 
\begin{align*}
  &\int_{\frac 1k}^{K}\dot \varphi_k  x u \left[u(-\Delta)^s \varphi_k-I(u, \varphi_k,\mathbb{R})(x)\right]dx = k \int_{\frac 1k}^{\frac2k}\dot \varphi(k \cdot)  x u \left[u(-\Delta)^s \varphi_k-I(u, \varphi_k,\mathbb{R})(x)\right]dx\\
  &=k^{2s-1} \int_{1}^{2}\dot \varphi  x u(\frac{\cdot}{k}) \left[u(\frac{\cdot}{k})(-\Delta)^s \varphi-I(u(\frac{\cdot}{k}), \varphi,\mathbb{R})(x)\right]dx,
\end{align*}
where, again thanks to Remark \ref{cut off} and similarly as before we have that 
\begin{align*}
\Bigl|\int_{1}^{2}\dot \varphi  x u(\frac{\cdot}{k}) \left[u(\frac{\cdot}{k})(-\Delta)^s \varphi-I(u(\frac{\cdot}{k}), \varphi,\mathbb{R})(x)\right]dx \bigr|
  \le C \int_{1}^{2}\Bigl(|(-\Delta)^s \varphi|+ \tilde I(\varphi)\Bigr) dx \bigr|< \infty
\end{align*}
with
$$
\tilde I(\varphi)(x) = \int_{\R} \frac{|\varphi(x)-\varphi(y)|}{|x-y|^{1+2s}}dy.
$$
Hence we also see that 
$$
\int_{\frac 1k}^{K}\dot \varphi_k  x u \left[u(-\Delta)^s \varphi_k-I(u, \varphi_k,\mathbb{R})(x)\right]dx \to 0
$$
and therefore
$$
\int_{\frac 1k}^{K} x\dot{(u {\varphi}_k)}\left[u(-\Delta)^s \varphi_k-I(u, \varphi_k,\mathbb{R})(x)\right]dx \to 0
$$
as $k \to +\infty$. Next we consider the integral
\begin{align}
  &\int_{-\infty}^{-\frac1k}x\mathcal N_s\left(u \varphi_k\right) \dot{\left(u \varphi_k\right)}dx \nonumber\\
  &=\int_{-\infty}^{-\frac1k}x\mathcal N_s\left(u \varphi_k\right) \dot u \varphi_k dx + \int_{-\frac2k}^{-\frac1k}x\mathcal N_s\left(u \varphi_k\right)u \dot \varphi_k dx \nonumber\\
  &=\int_{-\infty}^{-\frac1k}x \varphi_k (\mathcal N_s u) \dot u \varphi_k dx + \int_{-\frac2k}^{-\frac1k}x(\mathcal N_s u) \varphi_k u \dot \varphi_k dx \nonumber\\
&\hspace{1em}+ \int_{-\infty}^{-\frac1k}x  \Bigl(u\mathcal N_s \varphi_k -I(u,\varphi_k,\R^n_+)\Bigr) \dot u \varphi_k dx + \int_{-\frac2k}^{-\frac1k}x \Bigl(u\mathcal N_s \varphi_k -I(u,\varphi_k,\R^n_+)\Bigr)u \dot \varphi_k dx. \label{en-ex-v1-extra}
\end{align}
Again, we see, by Lebesgue's theorem, that 
$$
\int_{-\infty}^{-\frac1k} x \varphi_k (\mathcal N_s u) \dot u \varphi_k dx
\to \int_{-\infty}^{0} (\mathcal N_s u) x  \dot u dx \qquad \text{as $k \to +\infty$}
$$
and that, since $|\dot \varphi_k|\le C k$ on $[-\frac2k,-\frac1k]$ with a constant $C>0$
$$
\Bigl|\int_{-\frac2k}^{-\frac1k}x(\mathcal N_s u) \varphi_k u \dot \varphi_k dx\Bigr| \le C \int_{-\frac2k}^{-\frac1k}|\mathcal N_s u| dx \to 0
$$
as $k \to +\infty$, by Remark \ref{cut off}.

Now we deal with the two integrals in \eqref{en-ex-v1-extra}.
For the first one, arguing as in \eqref{first laplacian} we find 
\begin{equation}\label{first Neumann}\begin{aligned}
  &\Bigl|\int_{-K}^{-\frac1k}\varphi_k  x\dot{u} \left[u(\mathcal N_su) \varphi_k-I(u, \varphi_k,\mathbb{R^+})(x)\right]dx\Bigr|\\
  &=k^{2s-2} \Bigl|\int_{-kK}^{-1}  \varphi  x  \dot{u}(\frac{\cdot}{k}) \left[u(\frac{\cdot}{k})  (\mathcal N_s \varphi) -I(u(\frac{\cdot}{k}), \varphi,\mathbb{R}^+)(x)\right]dx\Bigr|\\
  &\le Ck^{2s-1} \int_{-kK}^{-1}  \varphi  \Bigl| u(\frac{\cdot}{k})  (\mathcal N_s \varphi) -I(u(\frac{\cdot}{k}), \varphi,\mathbb{R}^+)(x)\Bigr|dx.
\end{aligned}\end{equation}
Thanks to Remark \ref{cut off}, and using that $u\in L^\infty (\R)$ and  $\mathcal N_s\varphi \in L^1(\R)$, we find
$$
\int_{-k K}^{-1}  \varphi  \Bigl| u(\frac{\cdot}{k})  (\mathcal N_s \varphi)\Bigr| dx
\le C\int_{-\infty}^{-1}  \varphi  \Bigl|\mathcal N_s \varphi\Bigr| dx <\infty.
$$
Again by Remark \ref{cut off}, we have
$$
\Bigl| I(u(\frac{\cdot}{k}), \varphi,\mathbb{R^+})(x)\Bigr|\le 2\|u\|_{L^\infty(\R)} \int_{\R} \frac{|\varphi(x)-\varphi(y)|}{|x-y|^{1+2s}}dy
\le \frac{C}{(1+|x|)^{1+2s}}\quad \text{for $x<0$}
$$hence \[
\int _{-kK}^{-1}\varphi\Bigl|I(u(\frac{\cdot}{k}), \varphi,\mathbb{R}^+)(x)\Bigr|dx\leq C\int_{-\infty}^{-1}\frac{1}{(1+|x|)^{1+2s}}dx<\infty.
\]Since $s<\frac12$ the above estimates imply that the first integral in \eqref{en-ex-v1-extra} vanishes as $k\to+\infty.$

For the second one we observe that \begin{align*}
 \bigg| \int_{\frac 1k}^{K}\dot \varphi_k&  x u \left[u(\mathcal N_s\varphi_k)-I(u, \varphi_k,\mathbb{R}^+)(x)\right]dx\bigg| \\&\leq k \int_{\frac 1k}^{\frac2k}\left|\dot \varphi(k \cdot)  x u \left[u(\mathcal N_s \varphi_k)-I(u, \varphi_k,\mathbb{R}^+)(x)\right]\right|dx\\
  &=k^{2s-1} \int_{1}^{2}\left|\dot \varphi  x u(\frac{\cdot}{k}) \left[u(\frac{\cdot}{k})\mathcal N_s \varphi-I(u(\frac{\cdot}{k}), \varphi,\mathbb{R}^+)(x)\right]\right|dx\\
  &\leq C\int_{-2}^{-1}|\mathcal N_s \varphi+\widetilde I(\varphi)|dx
\end{align*}that is again finite by Remark \ref{cut off}, this also implies that the second integral in \eqref{en-ex-v1-extra} goes to $0$ as $k\to+\infty.$

It then remains to show that
\begin{equation}
\label{en-ex-v1-eq17}
\int_{T\left(\mathbb{R}^{+}\right)} \frac{\left(u \varphi_k(x)-u \varphi_k(y)\right)^2}{|x-y|^{1+2 s}} d x d y \to \int_{T\left(\mathbb{R}^{+}\right)} \frac{\left(u(x)-u(y)\right)^2}{|x-y|^{1+2 s}} d x d y.  
\end{equation}
To see this, we simply note that 
\begin{align*}
  \int_{T\left(\mathbb{R}^{+}\right)}& \frac{\left((u \varphi_k)(x)-(u \varphi_k)(y)\right)^2}{|x-y|^{1+2 s}} d x d y
  \\&=\int_{T\left(\mathbb{R}^{+}\right)} \frac{\Bigl(\varphi_k(x)(u(x)-u(y)) - u(y)(\varphi_k(x)-\varphi_k(y))\Bigr)^2}{|x-y|^{1+2 s}} d x d y\\
  &= \int_{T\left(\mathbb{R}^{+}\right)} \frac{\varphi_k^2(x)(u(x)-u(y))^2 }{|x-y|^{1+2s}}dxdy+\int_{T(\R^+)}\frac{ u^2(y)(\varphi_k(x)-\varphi_k(y))^2}{|x-y|^{1+2s}}dxdy\\&\hspace{5em}+\int_{T(\R^+)}\frac{2\varphi_k(x)(u(x)-u(y)) u(y)(\varphi_k(x)-\varphi_k(y))}{|x-y|^{1+2 s}} d x d y,
\end{align*}
where, by monotone convergence,
$$
\int_{T\left(\mathbb{R}^{+}\right)} \frac{\varphi_k^2(x)(u(x)-u(y))^2}{|x-y|^{1+2 s}} d x d y \to \int_{T\left(\mathbb{R}^{+}\right)} \frac{\left(u(x)-u(y)\right)^2}{|x-y|^{1+2 s}} d x d y
$$
as $k \to +\infty$ and, by a change of variables 
\begin{align*}
  \int_{T\left(\mathbb{R}^{+}\right)} \frac{u^2(y)(\varphi_k(x)-\varphi_k(y))^2}{|x-y|^{1+2 s}} d x d y&= k^{2s-1} \int_{T\left(\mathbb{R}^{+}\right)} \frac{u^2(ky)(\varphi(x)-\varphi(y))^2}{|x-y|^{1+2 s}} d x d y\\
  &\le k^{2s-1} \int_{\R^2} \frac{u^2(ky)(\varphi(x)-\varphi(y))^2}{|x-y|^{1+2 s}} d x d y \\&\le  k^{2s-1}\|u\|_{L^\infty(\R)}
\int_{\R^2} \frac{(\varphi(x)-\varphi(y))^2}{|x-y|^{1+2 s}} d x d y\\
&\le C k^{2s-1}   \int_{\R} \frac{1}{(1+|y|)^{1+2s}} d y \le Ck^{2s-1} \to 0 
  \end{align*}as $k\to+\infty.$

Finally, we also see that, by the Cauchy-Schwarz inequality,
  \begin{align*}
    &\Bigl|\int_{T\left(\mathbb{R}^{+}\right)} \frac{\varphi_k(x)(u(x)-u(y)) u(y)(\varphi_k(x)-\varphi_k(y))}{|x-y|^{1+2 s}} d x d y\Bigr|^2\\
    &\le \Bigl(\int_{T\left(\mathbb{R}^{+}\right)} \frac{\varphi_k^2(x)(u(x)-u(y))^2}{|x-y|^{1+2 s}} d x d y\Bigr) \Bigl(\int_{T\left(\mathbb{R}^{+}\right)} \frac{u^2(y)(\varphi_k(x)-\varphi_k(y))^2}{|x-y|^{1+2 s}} d x d y\Bigr)\\
    &\to 0
  \end{align*}
as $k \to +\infty$. Hence \eqref{en-ex-v1-eq17} follows, and the proof is finished.
\end{proof}

We can now establish the following
\begin{lemma}\label{energy expansion-v2}
  Let $s \in (0,\frac{1}{2})$, and let $u \in C^1_{\mathrm {loc}}(\R \setminus \{0\}) \cap L^\infty(\R_+)$ be a function satisfying the following assumptions for some $p \in (1,\frac{1+2s}{1-2s})$ and where $\beta =\min\{1,\frac{2s}{p-1}\}$:
  \begin{enumerate}[$(i)$]
   \item $(-\Delta)^s u \in L^1(0,\infty);$
   \item $\mathcal N_s u \in L^1(-\infty,0)$;
   \item $|u(x)| \le C \min \{1,|x|^{-\frac{2s}{p-1}}\}$ for $x \in (0,+\infty)$;
   \item $|\dot u(x)| \le C \frac{\min \{1,x^{-\frac{2s}{p-1}}\}}{x}$ for $x \in (0,+\infty);$
    \item $|u(x)| \le C \min \{1,|x|^{-\beta}\}$ for $x \in (-\infty,0)$;
    \item $|\dot u(x)|\le C \frac{\min \{1,|x|^{-\beta}\}}{|x|}$ for $x \in (-\infty,0).$
   \end{enumerate}
  Then \[
    (1-2s)\frac{c_s}{2}\iint_{T(\R^+)}\frac{(u(x)-u(y))^2}{|x-y|^{1+2s}}dxdy=-2\int_0^{+\infty}(-\Delta)^su(x)x\dot u(x)dx-2\int_{-\infty}^0\mathcal N_su(x)x\dot u(x)dx.
    \]
\end{lemma}

\begin{proof} We consider the functions $u\varphi_k$, $k \in \mathbb N$, where $\varphi_k = \varphi(\frac{\cdot}{k})$ and the function $\varphi$ is the one of Remark \ref{cut off}.
 Applying Lemma \ref{energy expansion-v1} to $u \varphi_k$, we obtain:
      \begin{align}
        (1-2 s) \frac{c_s}{2} &\int_{T\left(\mathbb{R}^{+}\right)} \frac{\left(u \varphi_k(x)-u \varphi_k(y)\right)^2}{|x-y|^{1+2 s}} d x d y
        \\&=\nonumber-2 \int_0^{+\infty}x(-\Delta)^s\left(u \varphi_k\right) \dot{\left(u \varphi_k\right)}dx -2  \int_{-\infty}^0x\mathcal N_s\left(u \varphi_k\right) \dot{\left(u \varphi_k\right)}dx \nonumber\\
&=-2 \int_{0}^{2k}x(-\Delta)^s\left(u \varphi_k\right) \dot{\left(u \varphi_k\right)}dx -2  \int_{-2k}^{0}x\mathcal N_s\left(u \varphi_k\right) \dot{\left(u \varphi_k\right)}dx. \label{en-ex-v2-eq1}          
\end{align}
      We can rewrite the first integral on the right hand side as
 \begin{align}
 &-2 \int_{0}^{2k}(-\Delta)^s\left(u \varphi_k\right) x\dot{(u \varphi_k)}dx \nonumber\\
&=-2 \int_{0}^{2k} x\varphi_k(-\Delta)^s u \dot{(u \varphi_k)}dx -2 \int_{0}^{2k} x\dot{(u {\varphi}_k)}\left[u(-\Delta)^s \varphi_k-I(u, \varphi_k,\mathbb{R})(x)\right]dx.\label{en-ex-v2-eq2}         
\end{align}
For the first integral in \eqref{en-ex-v2-eq2} we have
$$
-2 \int_{0}^{2k} x\varphi_k(-\Delta)^s u \dot{(u \varphi_k)}dx
=-2\int_{0}^{2k}x\varphi_k^2\dot u (-\Delta)^su \:dx-2\int_{k}^{2k} x u  \varphi_k \dot \varphi_k (-\Delta)^su \:dx,
$$
where, by assumptions $(i)$, $(iv)$ and Lebesgue's theorem,
\[
  \int_{0}^{2k}x\varphi_k^2 \dot u (-\Delta)^su \:dx \to 
 \int_{0}^{\infty} x \dot u (-\Delta)^su\:dx 
\]
as $k \to +\infty$. Moreover, since $|\dot \varphi_k|\le \frac{C}{k}$ on $[k,2k]$ with a constant $C>0$, we have
\[
 \Bigl| \int_{k}^{2k} x u  \varphi_k \dot \varphi_k (-\Delta)^su \:dx \Bigr| \le  2 C \int_{k}^{\infty} |u (-\Delta)^s u| \:dx \to 0
 \qquad \text{as $k \to +\infty$}
\]
since $u \in L^\infty(\R)$ and $(-\Delta)^s u \in L^1(0,\infty)$.
Hence 
\begin{equation}
\int_{0}^{2k} x\varphi_k(-\Delta)^s u \dot{(u \varphi_k)} \to   
\int_{0}^{\infty}x \dot u (-\Delta)^su \:dx \qquad \text{as $k \to +\infty$.} \label{en-ex-v2-eq3}         
\end{equation}
Next we consider the integral
\begin{align}
&\int_{0}^{2k} x\dot{(u {\varphi}_k)}\left[u(-\Delta)^s \varphi_k-I(u, \varphi_k,\mathbb{R})(x)\right]dx \label{en-ex-v2-eq4}\\
&= \int_{0}^{2k}x \varphi_k  \dot{u} \left[u(-\Delta)^s \varphi_k-I(u, \varphi_k,\mathbb{R})(x)\right]dx + \int_{0}^{2k}x \dot \varphi_k   u \left[u(-\Delta)^s \varphi_k-I(u, \varphi_k,\mathbb{R})(x)\right]dx \nonumber
\end{align}
We now use the fact that
\begin{equation}
  \label{en-ex-v2-eq4-1}
(-\Delta)^s \varphi_k  = k^{-2s} [(-\Delta)^s \varphi](\frac{\cdot}{k})
\end{equation}
to see that, by a change of variables,
\begin{align}
  \Bigl|&\int_{0}^{2k}  x\varphi_k \dot{u} u(-\Delta)^s \varphi_k\,dx \Bigr| \le \|(-\Delta)^s \varphi_k\|_{L^\infty(0,2k)} \int_{0}^{2k}  x | \dot{u} u|\,dx \nonumber\\
  &\le k^{-2s} \|(-\Delta)^s \varphi\|_{L^\infty(0,2)} \Bigl( \int_{0}^{1}   x | \dot{u} u|\,dx + \int_{1}^{2k} x | \dot{u} u|\,dx\Bigr)  \nonumber\\
  &\le C k^{-2s}  \Bigl( \int_{0}^{1} 1 \,dx + \int_{1}^{2k} x^{-\frac{4s}{p-1}}\,dx\Bigr)\\&\le C k^{-2s}  \Bigl(1 + k^{1-\frac{4s}{p-1}}\Bigr)  \nonumber \\
  &= C \Bigl(k^{-2s} + k^{1-2s\bigl(1+\frac{2}{p-1}\bigr)}\Bigr) \to 0 \qquad \text{as $k \to +\infty$,}\label{en-ex-v2-eq5} 
\end{align}
where we used assumption $(iii)$-$(iv)$ and  that $2s\bigl(1+\frac{2}{p-1}\bigr)>1$ since $p < \frac{1+2s}{1-2s}$ in the last step.

We also see that
\begin{align}
 \Bigl|\int_{0}^{2k}  x\varphi_k \dot{u} I(u, \varphi_k,\mathbb{R})(x)\,dx \Bigr| &\le \int_{0}^{2k}\min\{1,x^{-\frac{2s}{p-1}}\}|I(u, \varphi_k,\mathbb{R})(x)|dx\nonumber \\
&\le \int_{0}^{1}|I(u, \varphi_k,\mathbb{R})(x)|dx + \int_{1}^{2k} x^{-\frac{2s}{p-1}}|I(u, \varphi_k,\mathbb{R})(x)|dx,\label{en-ex-v2-eq6}
\end{align}
where
$$\begin{aligned}
|I(u, \varphi_k,\mathbb{R})(x)| &\le 2\|u\|_{L^\infty(\R)} \int_{\R}\frac{|\varphi(\frac{x}{k})-\varphi(\frac{y}{k})|}{|x-y|^{1+2s}}dy \\&= 2\|u\|_{L^\infty(\R)}k^{-2s} \int_{\R}\frac{|\varphi(\frac{x}{k})-\varphi(z)|}{|\frac{x}{k}-z|^{1+2s}}dz\\& \le C k^{-2s}\end{aligned}
$$
for $x \in (0,1)$ and therefore
$$
\int_{0}^{1}|I(u, \varphi_k,\mathbb{R})(x)|dx \to 0.
$$
Now we estimate, for $|x|\geq 1$, $I(u,\varphi_k,\R)(x)$.
By assumption $(iii)$, $(v)$ and a change of variables, 
\begin{align}
&\frac{1}{c_{N,s}}|I(u, \varphi_k,\mathbb{R})(x)|\le \nonumber\\  
 & \le \int_{\R} \frac{|u(x)-u(y)||\varphi_k(x)-\varphi_k(y)|}{|x-y|^{1+2 s}}dy
  \le C \int_{\R} \frac{(x^{-\beta}+\min \{1, |y|^{-\beta}\})|\varphi_k(x)-\varphi_k(y)|}{|x-y|^{1+2 s}}dy \nonumber\\
&= C k^{-2s} \int_{\R} \frac{(|x|^{-\beta}+\min \{1,|ky|^{-\beta}\})|\varphi(\frac{x}{k})-\varphi(y)|}{|\frac{x}{k}-y|^{1+2 s}}dy \nonumber \\
&= C k^{-2s} \int_{\{|y|\ge \frac{|x|}{k}\}} \frac{|x|^{-\beta}|\varphi(\frac{x}{k})-\varphi(y)|}{|\frac{x}{k}-y|^{1+2 s}}dy +C k^{-2s} \int_{\{|y|\le \frac{|x|}{k}\}} \frac{\min\{1,|ky|^{-\beta}\}|\varphi(\frac{x}{k})-\varphi(y)|}{|\frac{x}{k}-y|^{1+2 s}}dy  \nonumber \\ 
&\le C k^{-2s} |x|^{-\beta} \int_{\{|y|\ge \frac{|x|}{k}\}} \frac{|\varphi(\frac{x}{k})-\varphi(y)|}{|\frac{x}{k}-y|^{1+2 s}}dy +C k^{-2s} \int_{\{|y|\le \frac{|x|}{k}\}} \frac{\min\{1,|ky|^{-\beta}\}}{|\frac{x}{k}-y|^{2 s}}dy \nonumber \\ 
&\le C k^{-2s} |x|^{-\beta} \int_{\R} \frac{|\varphi(\frac{|x|}{k})-\varphi(y)|}{|\frac{x}{k}-y|^{1+2 s}}dy +C \int_{\{|y|\le \frac{|x|}{k}\}} \frac{\min\{1,|ky|^{-\beta}\}}{|x-k y|^{2 s}}dy \nonumber \\ 
&\le C k^{-2s} |x|^{-\beta}  +C k^{-1}\int_{\{|z|\le |x|\}} \frac{\min\{1,|z|^{-\beta}\}}{|x-z|^{2 s}}dz \nonumber \\ 
&= C k^{-2s} |x|^{-\beta}  +C k^{-1}\int_{\{|z|\le 1\}} |x-z|^{-2 s}dz +C k^{-1}\int_{\{1 \le |z|\le |x|\}} \frac{|z|^{-\beta}}{|x-z|^{2 s}}dz \nonumber \\ 
&\le C k^{-2s} |x|^{-\beta}  +C k^{-1}x^{-2s} +C k^{-1}|x|^{1-2s-\beta} \int_{\{\frac{1}{|x|} \le |t|\le 1\}} \frac{|t|^{-\beta}}{|1-t|^{2 s}}dz \label{special beta=1}.\end{align}
Now we distinguish between $\beta<1$ and $\beta=1$. If $\beta<1$ we can estimate the integral in \eqref{special beta=1} as \[
\int_{\{\frac{1}{|x|} \le |t|\le 1\}} \frac{|t|^{-\beta}}{|1-t|^{2 s}}dz\leq C\max\{1,|x|^{\beta-1}\}.
\]Consequently we have
\begin{align}  \\ 
\frac{1}{c_{N,s}}|I(u, \varphi_k,\mathbb{R})(x)|&\le C k^{-2s} |x|^{-\beta}  +C k^{-1}|x|^{-2s} +C k^{-1}|x|^{1-2s-\beta} \max \{1, |x|^{\beta-1}\} \nonumber\\ 
&\le C k^{-2s} |x|^{-\beta}  +C k^{-1}|x|^{-2s}\Bigl(1+ \max\{1,|x|^{1-\beta}\}\Bigr) \nonumber\\
&\le C \Bigl(k^{-2s} |x|^{-\beta}  + k^{-1}\Bigl(|x|^{-2s} +|x|^{1-2s-\beta}\Bigr)\bigr) \label{en-ex-v2-eq6-extra}.
\end{align}
Therefore,
\begin{align*}
  \int_{1}^{2k} x^{-\frac{2s}{p-1}}|I(u, \varphi_k,\mathbb{R})(x)|dx &\le  C \int_{1}^{2k} \Bigl(k^{-2s} x^{-\frac{2s}{p-1}-\beta}  + k^{-1}\Bigl(x^{-2s-\frac{2s}{p-1}} +x^{1-2s-\frac{2s}{p-1}-\beta}\Bigr)dx \\
  &\le C \Bigl(k^{-2s+1-\frac{2s}{p-1}-\beta}+ k^{-2s} +k^{-2s-\frac{2s}{p-1}}+k^{-1} \Bigr) \to 0
\end{align*}
as $k\to+\infty,$ using again the $2s\bigl(1+\frac{2}{p-1}\bigr)>1$ since $p < \frac{1+2s}{1-2s}$.

Now we deal with the case $\beta=1,$ starting from \eqref{special beta=1} we reach \[
\frac{1}{c_{N,s}}|I(u, \varphi_k,\mathbb{R})(x)|\le C k^{-2s} |x|^{-1}  +C k^{-1}|x|^{-2s} +C k^{-1}|x|^{-2s} \int_{\{\frac{1}{|x|} \le |t|\le 1\}} \frac{|t|^{-1}}{|1-t|^{2 s}}dz,
\]this time, integral on the right-hand side can be estimated as \[
\int_{\{\frac{1}{x} \le |t|\le 1\}} \frac{|t|^{-1}}{|1-t|^{2 s}}dt\leq C\max\{1,\log|x|\}
\]and hence we find \begin{align}\nonumber
\frac{1}{c_{N,s}}|I(u, \varphi_k,\mathbb{R})(x)|&\le C k^{-2s} x^{-1}  +C k^{-1}x^{-2s} +C k^{-1}x^{-2s}\max\{1,\log(|x|)\}\\
&\leq C(k^{-2s}x^{-1}+k^{-1}x^{-2s}\max\{1,\log(|x|)\}).\label{estimate beta=1}
\end{align}

Hence, we obtain \[\begin{aligned}
\int_{1}^{2k} x^{-\frac{2s}{p-1}}|I(u, \varphi_k,\mathbb{R})(x)|dx &\le C\bigg(\int_{1}^{2k}k^{-2s}x^{-\frac{2s}{p-1}-1}+k^{-1}x^{-\frac{2s}{p-1}-2s}\max\{1,\log(|x|)\}\bigg).\\
&\leq C (k^{-2s-\frac{2s}{p-1}}+k^{-2s}+k^{-\frac{2s}{p-1}-2s}\log(k)+k^{-1})\to 0
\end{aligned}
\]as $k\to+\infty.$

Inserting these estimates in \eqref{en-ex-v2-eq6} yields that 
\begin{equation}
\label{en-ex-v2-eq7}  
\int_{0}^{2k}  x\varphi_k \dot{u} I(u, \varphi_k,\mathbb{R})(x)\,dx \to 0 \qquad \text{as $k \to+ \infty$.}
\end{equation}
Combining \eqref{en-ex-v2-eq5} and \eqref{en-ex-v2-eq7}, we have proved that
\begin{equation}
  \label{en-ex-v2-eq8}
\int_{0}^{2k}x \varphi_k  \dot{u} \left[u(-\Delta)^s \varphi_k-I(u, \varphi_k,\mathbb{R})(x)\right]dx \to 0 \qquad \text{as $k \to+ \infty$.}
\end{equation}
Next, to estimate the integral
$$
\int_{0}^{2k} x \dot \varphi_k   u \left[u(-\Delta)^s \varphi_k-I(u, \varphi_k,\mathbb{R})(x)\right]dx = \int_{k}^{2k} x \dot \varphi_k   u \left[u(-\Delta)^s \varphi_k-I(u, \varphi_k,\mathbb{R})(x)\right]dx  
$$
we use again \eqref{en-ex-v2-eq4-1} to see that 
\begin{align*}
  \Bigl|\int_{k}^{2k}  x \dot \varphi_k u^2 (-\Delta)^s \varphi_k\,dx \Bigr| &\le \|(-\Delta)^s \varphi_k\|_{L^\infty(0,2k)} \int_{k}^{2k}  x | \dot{\varphi_k} u^2|\,dx \\
  &\le k^{-1-2s} \|(-\Delta)^s \varphi\|_{L^\infty(0,2)} \|\dot \varphi \|_{L^\infty(0,2)} \Bigl( \int_{k}^{2k}x u^2dx\Bigr)  \\
  &\le C k^{-1-2s}  \Bigl( \int_{k}^{2k} x^{1-\frac{4s}{p-1}} \,dx\Bigr)\\& \le C k^{1-2s\bigl(1+\frac{2}{p-1}\bigr)} \to 0 \qquad \text{as $k \to+ \infty$,}
\end{align*}
where we used again the fact that $2s\bigl(1+\frac{2}{p-1}\bigr)>1$ since $p < \frac{1+2s}{1-2s}$. We also see that, if $\beta>1$, by using \eqref{en-ex-v2-eq6-extra}, 
\begin{align*}
  \Bigl|&\int_{k}^{2k}  x \dot \varphi_k u I(u, \varphi_k,\mathbb{R})(x) \Bigr| \le 
   C k^{-1} \int_{k}^{2k} x^{1-\frac{2s}{p-1}}|I(u, \varphi_k,\mathbb{R})(x)|\,dx\\
   &\le C k^{-1} \int_{k}^{2k} x^{1-\frac{2s}{p-1}} \Bigl(k^{-2s}x^{-\beta} +  k^{-1}x^{-2s}+k^{-1} x^{1-2s-\beta}\Bigr)\,dx\\  
   &= C k^{-1-2s} \int_{k}^{2k} x^{1-\frac{2s}{p-1}-\beta}dx +Ck^{-1-2s}+k^{-2 }\int_{k}^{2k}  x^{2-2s-\frac{2s}{p-1}-\beta}\,dx\\  
  &\le  Ck^{-1-2s} + C k^{1-2s-\frac{2s}{p-1}-\beta} \to 0 \qquad \text{as $k \to +\infty$,}
  \end{align*}
recalling that, in this case, $\beta=\frac{2s}{p-1}$ and  $p<\frac{1+2s}{1-2s}$.
If $\beta=1$, we can argue in a similar way and use \eqref{estimate beta=1} to obtain\begin{align*}
  \Bigl|\int_{k}^{2k}  x \dot \varphi_k u I(u, \varphi_k,\mathbb{R})(x) \Bigr| &\le 
   C k^{-1} \int_{k}^{2k} x^{1-\frac{2s}{p-1}}|I(u, \varphi_k,\mathbb{R})(x)|\,dx\\
   &\le C k^{-1} \int_{k}^{2k} x^{1-\frac{2s}{p-1}} (k^{-2s}x^{-1}+Ck^{-1}x^{-2s}\log(|x|))dx\\  
   &\leq C(k^{-\frac{2s}{p-1}}+k^{-2s-1}\log(k))\to0\quad\text{as $k\to+\infty$.}
  \end{align*}
Combining these two estimates, we find that
\begin{equation}
  \label{en-ex-v2-eq9}
\int_{0}^{2k} x\dot{(u {\varphi}_k)}\left[u(-\Delta)^s \varphi_k-I(u, \varphi_k,\mathbb{R})(x)\right]dx \to 0 \qquad \text{as $k \to+ \infty$,}
\end{equation}both if $\beta=1$ or $\beta\in(0,1).$

From \eqref{en-ex-v2-eq2}, \eqref{en-ex-v2-eq3}, \eqref{en-ex-v2-eq4}, \eqref{en-ex-v2-eq8}, and \eqref{en-ex-v2-eq9}, we deduce that 
\begin{equation}
  \label{en-ex-v2-eq10}
-2 \int_{0}^{2k}(-\Delta)^s\left(u \varphi_k\right) x\dot{(u \varphi_k)}dx \to -2  \int_{0}^{\infty} x \dot u (-\Delta)^su\:dx
\qquad \text{as $k \to+ \infty$.}
\end{equation}
Next we consider the second integral in \eqref{en-ex-v2-eq1} given by 
\begin{align}\label{en-ex-v2-eq11}
 &\int_{-2k}^{0}x\mathcal N_s\left(u \varphi_k\right) \dot{\left(u \varphi_k\right)}dx \\
  &=\int_{-2k}^{0}x\mathcal N_s\left(u \varphi_k\right) \dot u \varphi_k dx + \int_{-2k}^{-k}x\mathcal N_s\left(u \varphi_k\right)u \dot \varphi_k dx \nonumber\\
  &=\int_{-2k}^{0} x \dot u   \varphi_k^2 (\mathcal N_s u)dx  + \int_{-2k}^{-k}x u (\mathcal N_s u) \varphi_k  \dot \varphi_k dx \nonumber\\ 
&\hspace{1em}+ \int_{-2k}^{0}x  \Bigl(u\mathcal N_s \varphi_k -I(u,\varphi_k,\R_+)\Bigr) \dot u \varphi_k dx + \int_{-2k}^{-k}x \Bigl(u\mathcal N_s \varphi_k -I(u,\varphi_k,\R_+)\Bigr)u \dot \varphi_k dx.\nonumber
\end{align}
Again, we see, by Lebesgue's theorem, using the facts that $\mathcal N_s u \in L^1(-\infty,0)$ and that $x \mapsto x \dot u$ is bounded on $(-\infty,0)$ by assumption,  we have that
\begin{equation}
\label{en-ex-v2-eq11-1}
\int_{-2k}^{0} x \dot u   \varphi_k^2 (\mathcal N_s u)dx \to \int_{-\infty}^{0} (\mathcal N_s u) x  \dot u dx \qquad \text{as $k \to +\infty$.}
\end{equation}
Moreover, since
$$
|u| \le C, \quad |\varphi_k| \le C \quad \text{and}\quad |\dot \varphi_k|\le \frac{C}{k} \qquad \text{on $[-2k,-k]$}
$$
with a constant $C>0$, we have 
\begin{equation}
\label{en-ex-v2-eq12}
\Bigl|\int_{-2k}^{-k}x(\mathcal N_s u) \varphi_k u \dot \varphi_k dx\Bigr| \le C \int_{-2k}^{-k} |\mathcal N_s u| dx \to 0
\end{equation}
as $k \to +\infty$. 

Now we wish to show that the last two integral in \eqref{en-ex-v2-eq11} vanish as $k\to+\infty$. We observe that
\begin{equation}
\label{en-ex-v2-eq13}
\mathcal N_s \varphi_k  = k^{-2s} [\mathcal N_s \varphi](\frac{\cdot}{k})
\end{equation}
and thanks to assumption $(v)$ we estimate
\begin{align*}
  \Bigl|\int_{-2k}^{-k}x \left(u\mathcal N_s \varphi_k \right)u \dot \varphi_k dx \Bigr|
  &\le C k^{-\beta}\int_{-2k}^{-k}| u\mathcal N_s \varphi_k |dx\\
  &= C k^{-1-\beta}\int_{-2k}^{-k}| u [\mathcal N_s \varphi](\frac{\cdot}{k}) |dx\\
  &= C k^{-1-\beta-2s}\int_{-2}^{-1} | u(k \cdot) \mathcal N_s \varphi| dx   \\
  &\leq C k^{-1-\beta-2s}\|u\|_{L^\infty(\R)}\|\mathcal N_s\varphi\|_{L^1(\R^-)}\to 0 \quad\text{as $k\to+\infty$}
,\end{align*}since $u$ is bounded and, by Remark \ref{cut off}, $\mathcal N_s \varphi \in L^1(\R^-).$
For the other term of the second integral in the last line of \eqref{en-ex-v2-eq11} we have, again by assumption $(v)$\begin{align}\label{integral R^- with beta}
   \bigg| \int_{-2k}^{-k}xu\dot\varphi_kI(u,\varphi_k,\R_+)dx\bigg|&\leq C\int_{-2k}^{-k}\min\{1,|x|^{-\beta}\}|I(u,\varphi_k,\R_+)|dx\\
   &= C\int_{-2k}^{-k}|x|^{-\beta}|I(u,\varphi_k,\R^+)|dx
\end{align}
if $k$ is big enough. By the estimate of \eqref{en-ex-v2-eq6-extra} we find, if $\beta<1$ and for $x\leq -1$, \[
|I(u,\varphi_k,\R^+)(x)|\leq C(k^{-2s}|x|^{-\beta}+k^{-1}(|x|^{-2s}+|x|^{1-2s-\beta}))
\]that combined with \eqref{integral R^- with beta} implies \begin{align*}
\bigg|\int_{-2k}^{-k}x u\dot{\varphi_k}I(u,\varphi_k,\R^+)dx\bigg|&\leq C\int_{-2k}^{-k}k^{-2s}|x|^{-2\beta}+k^{-1}|x|^{-2s-\beta}+k^{-1}|x|^{1-2s-2\beta}dx\\
&\leq C\left(k^{1-2s-2\beta}+k^{-2s-\beta}\right)
\end{align*}that vanishes as $k\to+\infty$ provided that $p<\frac{1+2s}{1-2s}.$ The case $\beta=1$ can be treated similarly using \eqref{estimate beta=1} instead of \eqref{en-ex-v2-eq6-extra}.
Consequently,
\begin{equation}
\label{en-ex-v2-eq14}
\int_{-2k}^{-k}x \Bigl(u\mathcal N_s \varphi_k -I(u,\varphi_k,\R_+)\Bigr)u \dot \varphi_k dx \to 0 \qquad \text{as $k \to +\infty$.}
\end{equation}
Moreover, we have
\begin{align*}
  &\Bigl|\int_{-2k}^{0}x  \Bigl(u\mathcal N_s \varphi_k -I(u,\varphi_k,\R_+)\Bigr) \dot u \varphi_k dx\Bigr|\\
  &\le C \int_{-2k}^0 \min\{1,|x|^{-\beta}\}\Bigl | u\mathcal N_s \varphi_k -I(u,\varphi_k,\R_+)\Bigr|\,dx\\   
  &\le Ck^{-2s} \int_{-2k}^0 \min\{1,|x|^{-\beta}\}|u [\mathcal N_s \varphi](\frac{\cdot}{k})|\,dx
   + C\int_{-2k}^0 \min\{1,|x|^{-\beta}\}|I(u,\varphi_k,\R_+)|dx \\
  &\le C k^{-2s} \|\mathcal N_s \varphi \|_{L^\infty(-\infty,0)} \int_{-2k}^0 \min\{1,|x|^{-2\beta}\}\,dx
    + C\int_{-2k}^0 \min\{1,|x|^{-\beta}\}|I(u,\varphi_k,\R_+)|dx,
\end{align*}
where
$$\begin{aligned}
k^{-2s} \|\mathcal N_s \varphi \|_{L^\infty(-\infty,0)} \int_{-2k}^0 \min\{1,|x|^{-2\beta}\}\,dx 
 &\le C k^{-2s}\Bigl(1+ k^{1-2\beta}\Bigr)\\&= C\Bigl(k^{-2s} + k^{1-2s-2\beta}\Bigr) \to 0,\end{aligned}
$$
using again that $2s\bigl(1+\frac{2}{p-1}\bigr)>1$ since $p < \frac{1+2s}{1-2s}$ and that $\mathcal N_s\varphi \in L^\infty(-\infty,0)$ by Remark \ref{cut off}. 

We also see that
$$
\int_{-2k}^0 \min\{1,|x|^{-\beta}\}|I(u,\varphi_k,\R_+)|dx \le \int_{-1}^{0}|I(u, \varphi_k,\mathbb{R})|dx + \int_{-2k}^{-1} |x|^{-\beta}|I(u, \varphi_k,\mathbb{R}_+)|dx,
$$
where for $x \in (-1,0)$ we have
$$
\begin{aligned}
|I(u, \varphi_k,\mathbb{R}_+)(x)| &\le 2\|u\|_{L^\infty(\R)} \int_{\R_+}\frac{|\varphi(\frac{x}{k})-\varphi(\frac{y}{k})|}{|x-y|^{1+2s}}dy \\&= 2\|u\|_{L^\infty(\R)}k^{-2s} \int_{\R_+}\frac{|\varphi(\frac{x}{k})-\varphi(z)|}{|\frac{x}{k}-z|^{1+2s}}dz \\&\le C k^{-2s}\end{aligned}
$$
and therefore
$$
\int_{-1}^{0}|I(u, \varphi_k,\mathbb{R})(x)|dx \to 0.
$$
By equation \eqref{en-ex-v2-eq6-extra}, we have if $\beta\in (0,1)$ 
\begin{align*}
  \int_{-2k}^{-1} |x|^{-\beta}|I(u, \varphi_k,\mathbb{R})(x)|dx &\le  C \int_{-2k}^{-1}
\Bigl(k^{-2s} |x|^{-2\beta}  + k^{-1}\Bigl( |x|^{-2s-\beta}+ |x|^{1-2s-2\beta}\Bigr)\Bigr)dx \\
  &\le C \Bigl(k^{-2s+1-2\beta}+ k^{-2s} +k^{-2s-\beta}+k^{-1} \Bigr) \to 0 \quad \text{as $k \to +\infty$,}
\end{align*}
using again the $2s\bigl(1+\frac{2}{p-1}\bigr)>1$ since $p < \frac{1+2s}{1-2s}$. Again if $\beta=1$, with the very same computations and using \eqref{estimate beta=1} instead of \eqref{en-ex-v2-eq6-extra}, we see also that \[
 \int_{-2k}^{-1} |x|^{-1}|I(u, \varphi_k,\mathbb{R})(x)|dx\to0\quad\text{as $k\to+\infty$}
.\]
From the estimates above, we infer that
\begin{equation}
\label{en-ex-v2-eq15}
\int_{0}^{2k}  x\varphi_k \dot{u} I(u, \varphi_k,\mathbb{R})(x)\,dx \to 0 \qquad \text{as $k \to +\infty$.}
\end{equation}
Combining \eqref{en-ex-v2-eq11}, \eqref{en-ex-v2-eq11-1}, \eqref{en-ex-v2-eq12}, \eqref{en-ex-v2-eq14} and \eqref{en-ex-v2-eq15}, 
we conclude that 
\begin{equation}
\label{en-ex-v2-eq16}
\int_{-2k}^{0}x\mathcal N_s\left(u \varphi_k\right) \dot{\left(u \varphi_k\right)}dx \to
\int_{-\infty}^{0} x  \dot{u} \mathcal N_su dx \qquad \text{as $k \to+ \infty$.}
\end{equation}
In view of \eqref{en-ex-v2-eq1}, \eqref{en-ex-v2-eq10} and \eqref{en-ex-v2-eq16}, it remains to show that
\begin{equation}
\label{en-ex-v2-eq17}
\int_{T\left(\mathbb{R}^{+}\right)} \frac{\left(u \varphi_k(x)-u \varphi_k(y)\right)^2}{|x-y|^{1+2 s}} d x d y \to \int_{T\left(\mathbb{R}^{+}\right)} \frac{\left(u(x)-u(y)\right)^2}{|x-y|^{1+2 s}} d x d y.  
\end{equation}
To see this, we simply note that 
\begin{align*}
  \int_{T\left(\mathbb{R}^{+}\right)} &\frac{\left((u \varphi_k)(x)-(u \varphi_k)(y)\right)^2}{|x-y|^{1+2 s}} d x d y
  \\&=\int_{T\left(\mathbb{R}^{+}\right)} \frac{\Bigl(\varphi_k(x)(u(x)-u(y)) - u(y)(\varphi_k(x)-\varphi_k(y))\Bigr)^2}{|x-y|^{1+2 s}} d x d y\\
  &= \int_{T\left(\mathbb{R}^{+}\right)} \frac{\varphi_k^2(x)(u(x)-u(y))^2}{|x-y|^{1+2s}}dxdy+\int_{T(\R^+)}\frac{u^2(y)(\varphi_k(x)-\varphi_k(y))^2}{|x-y|^{1+2s}}dxdy\\&\hspace{5em}+\int_{T(\R^+)}\frac{ 2\varphi_k(x)(u(x)-u(y)) u(y)(\varphi_k(x)-\varphi_k(y))}{|x-y|^{1+2 s}} d x d y,
\end{align*}
where, by monotone convergence,
$$
\int_{T\left(\mathbb{R}^{+}\right)} \frac{\varphi_k^2(x)(u(x)-u(y))^2}{|x-y|^{1+2 s}} d x d y \to \int_{T\left(\mathbb{R}^{+}\right)} \frac{\left(u(x)-u(y)\right)^2}{|x-y|^{1+2 s}} d x d y
$$
as $k \to +\infty$.

Moreover, by a change of variables and using Remark \ref{cut off}, we have 
\begin{align*}
  \int_{T\left(\mathbb{R}^{+}\right)} &\frac{u^2(y)(\varphi_k(x)-\varphi_k(y))^2}{|x-y|^{1+2 s}} d x d y= k^{1-2s} \int_{T\left(\mathbb{R}^{+}\right)} \frac{u^2(ky)(\varphi(x)-\varphi(y))^2}{|x-y|^{1+2 s}} d x d y\\
  &\le k^{1-2s} \int_{\R^2} \frac{u^2(ky)(\varphi(x)-\varphi(y))^2}{|x-y|^{1+2 s}} d x d y  \\&\le C k^{1-2s} \int_{\R} \frac{u^2(ky)}{(1+|y|)^{1+2s}} d y
   \\&\le C k^{1-2s} \int_{\R} \frac{\min \{1,|ky|^{-2\beta}\} }{(1+|y|)^{1+2s}} d y \\&= C k^{1-2s}\Bigl( \int_{\{|y|\le \frac{1}{k}\}} \frac{dy}{(1+|y|)^{1+2s}} + \int_{\{|y|> \frac{1}{k}\}} \frac{|ky|^{-2\beta}dy}{(1+|y|)^{1+2s}} \Bigr)\\
  &= C \Bigl(k^{-2s}  + k^{1-2s-2\beta}\int_{\{|y|> \frac{1}{k}\}}\frac{|y|^{-2\beta}}{(1+|y|)^{1+2s}} d y\Bigr)\\& \le C \Bigl(k^{-2s}  + k^{1-2s-2\beta}(1+ k^{2\beta-1})\Bigr)\to 0 \qquad \text{as $k \to +\infty$,}  
  \end{align*}
recalling that $\beta=\min\{1,\frac{2s}{p-1}\}$ and  $1<p<\frac{1+2s}{1-2s}.$
  
  Finally, we also see that, by the Cauchy-Schwarz inequality,
  \begin{align*}
    &\Bigl|\int_{T\left(\mathbb{R}^{+}\right)} \frac{\varphi_k(x)(u(x)-u(y)) u(y)(\varphi_k(x)-\varphi_k(y))}{|x-y|^{1+2 s}} d x d y\Bigr|^2\\
    &\le \Bigl(\int_{T\left(\mathbb{R}^{+}\right)} \frac{\varphi_k^2(x)(u(x)-u(y))^2}{|x-y|^{1+2 s}} d x d y\Bigr) \Bigl(\int_{T\left(\mathbb{R}^{+}\right)} \frac{u^2(y)(\varphi_k(x)-\varphi_k(y))^2}{|x-y|^{1+2 s}} d x d y\Bigr)\to 0
  \end{align*}
as $k \to +\infty$. Hence \eqref{en-ex-v2-eq17} follows, and the proof is finished.
\end{proof}
From the previous result we immediately obtain our Pohozaev identity.
\begin{proof}[Proof of Theorem \ref{Pohozaev}]
The proof just follows by combining the decay estimates of Section \ref{sec 3 lane emden} and Lemma \ref{energy expansion-v2}.
    \end{proof}
Finally, we can give the proof of our Liouville-type result.
\begin{proof}[Proof of Theorem \ref{non exitence result}]
By Corollary~\ref{decay2}, $u$ satisfies the assumptions of Lemma~\ref{energy expansion-v2}. Consequently,
\begin{align*}
(1-2s)\int_{0}^{+\infty}u^{p+1}\,dx &=  (1-2s)\frac{c_s}{2}\iint_{T(\R^+)}\frac{(u(x)-u(y))^2}{|x-y|^{1+2s}}dxdy\\&=
                           -2\int_0^{+\infty}(-\Delta)^su(x)x\dot u(x)dx
                             =-2  \int_{0}^{+\infty}x u^p \dot u \,dx\\& = -\frac{2}{p+1}  \int_{0}^{+\infty}x \partial_x[u^{p+1}]\,dx
                             = \frac{2}{p+1}  \int_{0}^{+\infty}u^{p+1}\,dx.
\end{align*}
Suppose by contradiction that $u \not \equiv 0$, it follows that $(1-2s)=\frac{2}{p+1}$, hence
$$
p = \frac{2}{1-2s}-1= \frac{1+2 s}{1-2 s},
$$
contrary to the assumption. The proof is thus finished.
\end{proof}
\begin{ack}
E.C. and M.T. are members of the Gruppo Nazionale per l’Analisi Matematica, la Probabilità e le loro Applicazioni (GNAMPA) of the Istituto Nazionale di Alta Matematica (INdAM). E.C. and M.T. were partially funded by the INdAM–GNAMPA project "Problemi di ottimizzazione di forma in contesti anisotropi e non-locali
", CUP: E53C25002010001, and by the PRIN project 2022R537CS "NO$^3$--Nodal Optimization, NOnlinear elliptic equations, NOnlocal geometric problems, with a focus on regularity", CUP: J53D23003850006. Part of the project has
been finalized during a three-month stay of M.T. at the Institute f\"ur Mathematics of the Goethe-Universit\"at Frankfurt. The hosting institution is kindly acknowledged.
\end{ack}
\bibliographystyle{abbrv}
\bibliography{biblio}
\end{document}